\documentclass{amsart}
\usepackage{amssymb}
\usepackage[all,cmtip]{xy}
\input diagxy
\usepackage[backref]{hyperref}

\newtheorem{theorem}{Theorem}[section]
\newtheorem{proposition}[theorem]{Proposition}
\newtheorem{corollary}[theorem]{Corollary}
\newtheorem{lemma}[theorem]{Lemma}

\theoremstyle{definition}
\newtheorem{definition}[theorem]{Definition}
\newtheorem{example}[theorem]{Example}
\newtheorem{axiom}{Axiom}

\numberwithin{equation}{section}

\newtheorem{remark}[theorem]{Remark}

\DeclareMathOperator{\sk}{sk}
\DeclareMathOperator{\Sk}{Sk}
\DeclareMathOperator{\csk}{cosk}
\DeclareMathOperator{\Csk}{Cosk}
\DeclareMathOperator*{\colim}{colim}
\DeclareMathOperator{\Dec}{Dec}
\DeclareMathOperator{\tr}{tr}
\DeclareMathOperator{\pr}{pr}

\newcommand{\too}{\longrightarrow}
\newcommand{\sm}{\mathsf{Smooth}}
\newcommand{\ssm}{\mathsf{sSmooth}}
\newcommand{\st}{\mathsf{Set}}
\newcommand{\sst}{\mathsf{sSet}}
\newcommand{\snst}{\mathsf{s}_{\leq n}\mathsf{Set}}
\newcommand{\C}{\mathsf{C}}
\newcommand{\sC}{\mathsf{sC}}
\newcommand{\snC}{\mathsf{s}_{\leq n}\mathsf{C}}
\newcommand{\tn}{\tau_n}
\newcommand{\sgp}{\mathsf{sGroup}}
\newcommand{\W}{W}
\newcommand{\LW}{\overline{W}}
\newcommand{\spn}{\textup{Spin}(n)}
\newcommand{\strn}{\textup{String}(n)}
\newcommand{\fivebrn}{\textup{Fivebrane}(n)}
\newcommand{\p}{d}
\newcommand{\s}{s}
\newcommand{\G}{\mathcal{G}}
\newcommand{\hooklongrightarrow}{\lhook\joinrel\longrightarrow}
\newcommand{\g}{{\boldsymbol g}}
\newcommand{\h}{{\boldsymbol h}}

\title{Descent for $n$-bundles}

\author{Jesse Wolfson}

\email{wolfson@math.uchicago.edu}

\address{University of Chicago}

\thanks{This work was partially supported by an
  NSF Graduate Research Fellowship under Grant No.\ DGE-0824162, and by
  an NSF Research Training Group in the Mathematical Sciences under
  Grant No.\ DMS-0636646.}

\begin{document}

\begin{abstract}
  Given a Lie group $G$, one constructs a principal $G$-bundle on a
  manifold $X$ by taking a cover $U\rightarrow X$, specifying a
  transition cocycle on the cover, and descending the trivialized
  bundle $U\times G$ along the cover.  We demonstrate the existence of
  an analogous construction for local $n$-bundles for general $n$.  We
  establish analogues for simplicial Lie groups of Moore's results
  on simplicial groups; these imply that bundles for strict Lie
  $n$-groups arise from local $n$-bundles.  Our construction leads
  to simple finite dimensional models of Lie 2-groups such as $\strn$ from cocycle data.
\end{abstract}

\maketitle

\section{Introduction}
The nerve of a group is a simplicial set satisfying Kan's horn-filling conditions.  Grothendieck observed that the nerve provides an equivalence between the category of groups and the category of reduced Kan simplicial sets whose horns have unique fillers above dimension one. More generally, he showed that the nerve extends to an equivalence between the category of groupoids and the category of Kan simplicial sets whose horns have unique fillers above dimension one. Inspired by this, Duskin \cite{Dus:79} defined an $n$-groupoid to be a Kan simplicial set whose horns have unique fillers above dimension $n$.

In the last decade, Henriques \cite{Hen:08}, Pridham \cite{Pri:13} and others have begun the study of Lie $n$-groupoids: simplicial manifolds whose horn-filling maps are surjective submersions in all dimensions, and isomorphisms above dimension $n$.  Examples are common. (See also Getzler \cite{Get:09}.)  Lie $0$-groupoids are precisely smooth manifolds.  Lie 1-groups are nerves of Lie groups.  Abelian Lie $n$-groups are equivalent to chain complexes of abelian Lie groups supported between degrees $0$ and $n-1$.

Simplicial Lie groups whose underlying simplicial set is an $(n-1)$-groupoid give rise to Lie $n$-groups, by the $\LW$-construction (Section 6). We call this special class of Lie $n$-groups \emph{strict Lie $n$-groups}.

Much of the theory of principal bundles for Lie groups generalizes naturally to principal bundles for strict Lie $n$-groups.  In Theorems \ref{thm:desc} and \ref{thm:wbar}, we show that the construction of a fiber bundle from a cocycle on a cover has a close analogue for cocycles for strict Lie $n$-groups.  As an application, we show how this allows for the construction of finite dimensional Lie 2-groups, such as $\strn$, from cohomological data.  The method works equally well for $n>2$.

\subsection*{Outline}
We develop our results within a category $\C$ which has a terminal object $\ast$ and a subcategory of covers.  We require the subcategory of covers to be stable under pullback,  to contain the maps $X\to\ast$ for every object $X$, to satisfy an axiom of right-cancelation, and to be contained within the class of effective epimorphisms. Our motivating example is the category of finite dimensional smooth manifolds, with surjective submersions as covers.  Other examples include Banach manifolds, analytic manifolds over a complete normed field, and of course sets.

In Section \ref{sec:stacks}, we recall the definitions and basic properties of higher stacks. This section parallels the discussion in Behrend and Getzler \cite{BeG:13}, the main difference being that in that paper, the category $\C$ is assumed to possess finite limits. Here, we work with categories of manifolds, so this assumption does not hold. ( Conversely, they do not impose the assumption that the maps $X\to\ast$ are covers, which fails in the setting of not-necessarily smooth analytic spaces.)

In Section \ref{sec:highhom}, we show that the collection of $k$-morphisms in a Lie $n$-groupoid forms a Lie $(n-k)$-groupoid (Theorem \ref{thm:kmor}).  In Section \ref{sec:stric}, we apply this to study Duskin's $n$-strictification functor $\tn$. (For a discussion in the absolute case, see \cite[Section~3.1, Example~5]{Gle:82}, \cite[Definition~3.5]{Hen:08} or \cite[Section~2]{Get:09}.)  In the Lie setting, the functor $\tn$ does not always exist.  However, when it does, it provides a partial left adjoint to the inclusion of $n$-stacks into the category of $\infty$-stacks.  We recall the relevant properties and give a necessary and sufficient criterion for existence.  In Section \ref{sec:descent}, we impose the additional axiom that quotients of regular equivalence relations in $\C$ exist. (This was first established by Godemont for smooth manifolds, or analytic manifolds over a complete normed field.) Under this assumption, we introduce local $n$-bundles and prove our main result on descent (Theorem \ref{thm:desc}).

In Section \ref{sec:slie}, we extend results of Moore on simplicial groups in $\st$ to simplicial Lie groups. These results provide a ready supply of examples satisfying the hypotheses of Theorem \ref{thm:desc}. We conclude in Section \ref{sec:string}, by applying our results to construct finite dimensional Lie 2-groups from cocycle data.  We describe the resulting model of $\strn$ and compare it to the model constructed by Schommer-Pries \cite{Sch:11}.

\subsection*{Acknowledgments}
The results of Section \ref{sec:slie} represent joint work with E. Getzler.  The author thanks him and J. Batson for helpful comments on several drafts. The author also thanks the anonymous referee for numerous helpful remarks which substantially improved the exposition.

\section{Higher Stacks}\label{sec:stacks}

We work in a category $\C$ with a subcategory of ``covers''.

\begin{axiom}\label{axiom:fib}
  The category $\C$ has a terminal object $\ast$, and the map
  $X\to\ast$ is a cover for every object $X\in\C$.
\end{axiom}

\begin{axiom}\label{axiom:top}
  Pullbacks of covers along arbitrary maps exist and are covers.
\end{axiom}

\begin{axiom}\label{axiom:fg+g}
  If $g$ and $f$ are composable maps such that $fg$ and $g$ are
  covers, then $f$ is also a cover.
\end{axiom}

If it exists, the \emph{kernel pair} of a map $f:X\to Y$ in a category $\C$ is
the pair of parallel arrows
\begin{equation*}
  \begin{xy}
    \morphism/{@{>}@<3pt>}/<600,0>[X\times_YX`X;]
    \morphism/{@{>}@<-3pt>}/<600,0>[X\times_YX`X;]
  \end{xy}
\end{equation*}
The map $f:X\to Y$ is an \emph{effective epimorphism} if $f$ is
the coequalizer of this pair:
\begin{equation*}
  \begin{xy}
    \morphism/{@{>}@<3pt>}/<600,0>[X\times_YX`X;]
    \morphism/{@{>}@<-3pt>}/<600,0>[X\times_YX`X;]
    \morphism(600,0)<500,0>[X`Y;f] .
  \end{xy}
\end{equation*}

\begin{axiom}\label{axiom:subcan}
  Covers are effective epimorphisms.
\end{axiom}

Axioms \ref{axiom:fib} and \ref{axiom:top} ensure that isomorphisms
are covers, that $\C$ has finite products, that projections along
factors of products are covers, and that covers form a pre-topology on
$\C$. Axiom \ref{axiom:fg+g}, which we borrow from Behrend and
Getzler \cite{BeG:13}, ensures that being a cover is a local property:
it is preserved \emph{and} reflected under pullback along
covers. Likewise, Axiom \ref{axiom:subcan} ensures that being an
isomorphism is a local property.

We will be interested in the following examples of categories
  with covers satisfying these assumptions:
\begin{enumerate}
\item $\st$, the category of sets, with surjections as covers;
\item $\sm$, the category of finite-dimensional smooth manifolds, with
  surjective submersions as covers;
\item the category of Banach manifolds, with surjective submersions as
  covers (see \cite{Hen:08});
\item the category of analytic manifolds over a complete normed field,
  with surjective submersions as covers (see \cite[Chapter
  III]{Ser:64}).
\end{enumerate}

Denote by $\sC$ the category of simplicial objects in $\C$. In
particular, we have the category of simplicial sets $\sst$. The
category $\C$ embeds fully faithfully in $\sC$ as the category of constant simplicial diagrams.  We do not distinguish between the category $\C$ and its essential image under this embedding.

Let $\Delta^k$ denote the standard $k$-simplex, that is, the simplicial set $\Delta^k=\Delta(-,[k])$. A simplicial set $S$ is the colimit of its simplices
\begin{equation*}
    \colim_{\Delta^k\to S} \Delta^k\cong S
\end{equation*}

\begin{definition}
    Let $X_\bullet$ be a simplicial object in $\C$. Let $S$ be a simplicial set.  Denote by $\hom(S,X)$ the limit
    \begin{equation*}
        \hom(S,X):=\lim_{\Delta^k\to S} X_k
    \end{equation*}
\end{definition}
Note that such limits do not exist in general.

By a Lie group, we will mean a group internal to the category $\C$,
that is, an object $G$ with product $m:G\times G\to G$, inverse
$i:G\to G$, and identity $e:\ast\to G$, satisfying the usual
axioms. We may associate to a Lie group its \emph{nerve} $N_\bullet G\in\sC$,
which is the simplicial object
\begin{equation*}
  N_kG = G^k .
\end{equation*}
The face maps $\p_i:G^k\to G^{k-1}$ are defined for $i=0$ and $i=k$ by
projection away from the first and last factor respectively, and for
$0<i<k$ by
\begin{equation*}
  \p_i = G^{i-1} \times m \times G^{k-1} .
\end{equation*}
The degeneracy maps $\s_i:G^k\to G^{k+1}$ are defined by
\begin{equation*}
  \s_i = G^i \times e \times G^{k-i} .
\end{equation*}
In fact, the above construction does not use the existence of an
inverse, and works if $G$ is only a monoid in $\C$.

We will use the following simplicial subsets of $\Delta^k$:
\begin{enumerate}
\item the boundary $\partial\Delta^k$ of $\Delta^k$;
\item the $i^{\textit{th}}$ horn
  $\Lambda^k_i\subset\partial\Delta^k$, obtained from
  $\partial\Delta^k$ by omitting its $i^{\textit{th}}$ face.
\end{enumerate}

\begin{definition}
  Let $f:X\to Y$ be a map in $\sC$.  The \emph{matching object} $M_k(f)$ is the limit
  \begin{equation*}
    \hom(\partial\Delta^k,X) \times_{\hom(\partial\Delta^k,Y)} Y_k .
  \end{equation*}
  Denote by $\mu_k(f)$ the induced map from $X_k$ to
  $M_k(f)$.

  The object of \emph{relative $\Lambda^k_i$-horns} $\Lambda^k_i(f)$ is the limit
  \begin{equation*}
    \hom(\Lambda^k_i,X) \times_{\hom(\Lambda^k_i,Y)} Y_k .
  \end{equation*}
  Denote by $\lambda^k_i(f)$ the induced map from
  $X_k$ to $\Lambda^k_i(f)$.
\end{definition}

A section of $M_k(f)$ is a $k$-simplex of $Y$ together with a lift of
its boundary to $X$, and $\mu_k(f)$ measures the extent to which these
relative spheres are filled by $k$-simplices of $X$. Similarly, a section of $\Lambda^k_i(f)$ is a $k$-simplex of $Y$ together with a
lift of the $\Lambda^k_i$-horn to $X$, and $\lambda^k_i(f)$ measures
the extent to which these relative horns are filled by $k$-simplices
of $X$.

In the absolute case, where the target $Y$ of the simplicial map $f$
is the terminal object $\ast$, we write $M_k(X)$ and
$\Lambda^k_i(X)$ instead of $M_k(f)$ and $\Lambda^k_i(f)$, and
similarly for the induced maps $\mu_k(X)$ and $\lambda^k_i(X)$.

As an example, we have
\begin{equation*}
  \Lambda^k_i(N_\bullet G) \cong
  \begin{cases}
    \ast , & k=0,1 , \\
    G^k , & k>1 .
  \end{cases}
\end{equation*}
This is easily seen if $0<i<k$; for $i=0$ or $i=k$, the proof requires
the existence of the inverse for $G$. In fact, the isomorphisms
$\Lambda^k_i(N_\bullet G)\cong N_kG$, $k>1$, together with the
condition $N_0G\cong\ast$, characterize the nerves of groups, and
indeed give an alternative axiomatization of the theory of groups.

Grothendieck extended this observation, omitting the condition
$N_0G\cong\ast$.
\begin{definition}
  A Lie groupoid $\G$ in $\C$ is an internal groupoid in $\C$, with
  morphisms $\G_1$ and objects $\G_0$, source and target maps
  $s,t:\G_1\to\G_0$, multiplication
  \begin{equation*}
    m\colon\G_1\times_{\G_0}^{t,s}\G_1\to \G_1,
  \end{equation*}
  unit $e:\G_0\to\G_1$, and inverse $i:\G_1\to\G_1$, such that $s$ and $t$ are covers.
\end{definition}
The \emph{nerve} $N_\bullet\G$ of a groupoid is the simplicial object
$N_\bullet \G\in\sC$,
\begin{equation*}
  N_k\G =
  \begin{cases}
    \G_0 , & k=0 , \\
    \G_1 , & k=1 , \\
    \underbrace{\G_1\times_{\G_0}^{t,s}\cdots\times_{\G_0}^{t,s}\G_1}_k , & k>1 .
  \end{cases}
\end{equation*}
On 1-simplices, the face maps $\p_0$ and $\p_1$ correspond to the target $t$ and source $s$.  The degeneracy $\s_0\colon\G_0\to\G_1$ corresponds to the unit.  On $k$-simplices for $k>1$, the face maps
\begin{equation*}
    \p_i:\underbrace{\G_1\times_{\G_0}^{t,s}\cdots\times_{\G_0}^{t,s}\G_1}_k\to\underbrace{\G_1\times_{\G_0}^{t,s}\cdots\times_{\G_0}^{t,s}\G_1}_{k-1}
\end{equation*}
are defined for $i=0$ and $i=k$ by projection away from the first and last factor respectively, and for
$0<i<k$ by
\begin{equation*}
  \p_i = \underbrace{\G_1\times_{\G_0}^{t,s}\cdots\times_{\G_0}^{t,s}\G_1}_{i-1}\times m \times \underbrace{\G_1\times_{\G_0}^{t,s}\cdots\times_{\G_0}^{t,s}\G_1}_{k-i} .
\end{equation*}
The degeneracy maps
\begin{equation*}
    \s_i:\underbrace{\G_1\times_{\G_0}^{t,s}\cdots\times_{\G_0}^{t,s}\G_1}_k\to\underbrace{\G_1\times_{\G_0}^{t,s}\cdots\times_{\G_0}^{t,s}\G_1}_{k+1}
\end{equation*}
are defined by
\begin{equation*}
  \s_i = \underbrace{\G_1\times_{\G_0}^{t,s}\cdots\times_{\G_0}^{t,s}\G_1}_{i-1} \times e \times \underbrace{\G_1\times_{\G_0}^{t,s}\cdots\times_{\G_0}^{t,s}\G_1}_{k-i} .
\end{equation*}
Grothendieck's observation, generalized to Lie groupoids from his
setting of discrete groupoids, is as follows.
\begin{proposition}[Grothendieck]
  A simplicial object $X_\bullet\in\sC$ is isomorphic to the nerve
  of a Lie groupoid if and only if the horn-filler maps
  \begin{equation*}
    \lambda^k_i(X) : X_k \to \Lambda^k_i(X)
  \end{equation*}
  are covers for $k=1$, and isomorphisms for $k>1$.

  In particular, a simplicial object $X_\bullet\in\sC$ is isomorphic
  to the nerve of a Lie group if and only if the above conditions are
  fulfilled and $X_0=\ast$.
\end{proposition}

Maps between Lie groupoids are in bijection with simplicial
maps between their nerves.  As a result, the full subcategory of $\sC$ consisting of
those simplicial objects satisfying the above conditions is equivalent
to the category of Lie groupoids.

Motivated by Grothendieck's observation, Duskin \cite{Dus:79}
introduced a notion of $n$-groupoid valid in any topos.  Duskin's notion was adapted
by Henriques \cite{Hen:08} (see also \cite{Get:09}) to
cover higher Lie groupoids.
\begin{definition}
  \label{defi:ngpd}
  Let $n\in\mathbb{N}\cup\{\infty\}$. A \emph{Lie $n$-groupoid} is a
  simplicial object $X_\bullet\in\sC$ such that for all $k>0$ and $0\le i\le
  k$, the limit $\Lambda^k_i(X)$ exists in $\C$, the map
  \begin{equation*}
    \lambda^k_i(X)\colon X_k \to \Lambda^k_i(X)
  \end{equation*}
  is a cover, and it is an isomorphism for $k>n$.

  A \emph{Lie $n$-group} $X_\bullet$ is a Lie $n$-groupoid such that
  $X_0=\ast$.
\end{definition}

As an example, a Lie 0-groupoid is the same as an object of $\C$ (viewed as a constant simplicial diagram).

\begin{definition}[Verdier]
  \label{def:hyp}
  Let $n\in\mathbb{N}\cup\{\infty\}$. A map $f:X_\bullet\to
  Y_\bullet$ of Lie $\infty$-groupoids is an \emph{$n$-hypercover}
  if, for all $k\ge0$, the limit $M_k(f)$ exists in $\C$, the
  map
  \begin{equation*}
    \mu_k(f)\colon X_k \to M_k(f)
  \end{equation*}
  is a cover for all $k$, and it is an isomorphism for $k\geq n$.
\end{definition}

Hypercovers in $\sst$ are the same as trivial fibrations, that is,
Kan fibrations which are also weak homotopy equivalences. Hypercovers play
much the same role in the theory of Lie $n$-groupoids. We refer
to an $\infty$-hypercover simply as a ``hypercover.''

A $0$-hypercover is an isomorphism, while a $1$-hypercover of a Lie
$0$-groupoid is isomorphic to the nerve of the cover $f_0:X_0\to Y_0$. In
other words,
\begin{equation*}
  X_k \cong \underbrace{X_0\times_{Y_0}\times\cdots\times_{Y_0}X_0}_{k+1}
\end{equation*}

\begin{definition}
    An \emph{augmentation} of a simplicial object $X_\bullet\in\sC$ is a
    simplicial map to an object $Y\in\C\subset\sC$.
\end{definition}
This amounts to the same thing as a map $\varepsilon:X_0\too Y$
that renders the diagram
\begin{equation*}
  \begin{xy}
    \morphism|a|/@{>}@<3pt>/<400,0>[X_1`X_0;\p_0]
    \morphism|b|/@{>}@<-3pt>/<400,0>[X_1`X_0;\p_1]
    \morphism(400,0)<400,0>[X_0`Y;\varepsilon]
  \end{xy}
\end{equation*}
commutative.

\begin{definition}
    The \emph{orbit space} $\pi_0(X)$ of a Lie $\infty$-groupoid
    $X_\bullet$ is a cover
    \begin{equation*}
        X_0\too\pi_0(X)
    \end{equation*}
    which coequalizes the fork
    \begin{equation*}
      \begin{xy}
        \morphism|a|/@{>}@<3pt>/<400,0>[X_1`X_0;\p_0]
        \morphism|b|/@{>}@<-3pt>/<400,0>[X_1`X_0;\p_1]
        \morphism(400,0)<400,0>[X_0`\pi_0(X);]
      \end{xy}
    \end{equation*}
\end{definition}
In other words, for any augmentation $\varepsilon:X_0\too Y$ of $X_\bullet$, there is an induced map
\begin{equation*}
    \begin{xy}
      \morphism|a|/@{>}@<3pt>/<400,0>[X_1`X_0;\p_0]
      \morphism|b|/@{>}@<-3pt>/<400,0>[X_1`X_0;\p_1]
      \morphism(400,0)<400,0>[X_0`\pi_0(X);]
      \morphism(800,0)/{.>}/<0,-400>[\pi_0(X)`Y;]
      \morphism(400,0)<400,-400>[X_0`Y;]
    \end{xy}
\end{equation*}
Furthermore, if $\varepsilon$ is a cover, then so is the induced map
from $\pi_0(X)$ to $Y$, by Axiom~\ref{axiom:fg+g}.

It is characteristic of the theory of Lie $\infty$-groupoids that
the orbit space $\pi_0(X)$ need not exist. One case in which the
orbit space of $X_\bullet$ exists, however, is when $X_\bullet$ admits an augmentation
$\varepsilon:X_\bullet\to Y$ which is a hypercover.
\begin{proposition}\label{prop:hyppi0}
  \mbox{}
  \begin{enumerate}
  \item An augmentation $\varepsilon:X_\bullet\to Y$ is an
    $n$-hypercover if and only if the maps $\varepsilon:X_0\to Y$ and
    $\mu_1(\varepsilon):X_1\to X_0\times_YX_0$ are covers, and the
    maps
    \begin{equation*}
        \mu_k(X)\colon X_k\to M_k(X)
    \end{equation*}
    are covers for $k>1$ and isomorphisms for $k\ge n$.
  \item If the augmentation $\varepsilon:X_\bullet\to Y$ is a
    hypercover, then $\pi_0(X)\cong Y$.
  \end{enumerate}
\end{proposition}
\begin{proof}
  If $\varepsilon:X_\bullet\to Y$ is an augmentation, we have
  \begin{equation*}
    M_k(\varepsilon) =
    \begin{cases}
      Y , & k=0 , \\
      X_0\times_YX_0 , & k=1 , \\
      M_k(X) , & k>1 .
    \end{cases}
  \end{equation*}
  The first part follows by inspection.

  The second part is a restatement of Axiom~\ref{axiom:subcan}.
\end{proof}

Let $\Delta_{\le n}\subseteq\Delta$ be the full subcategory with
objects $\{[m] \mid m\le n\}$. An \emph{$n$-truncated} simplicial
object is a functor
\begin{equation*}
  X_{\le n} : \Delta_{\leq n}^\circ \to \C .
\end{equation*}
Denote by $\snC$ the category of $n$-truncated simplicial objects in
$\C$. Restriction along $\Delta_{\leq n}\hookrightarrow\Delta$ induces
the functor of \emph{$n$-truncation}:
\begin{equation*}
  \tr_n : \sC \to \snC .
\end{equation*}
When $S$ is a simplicial set of dimension less than or equal to $n$,
we abuse notation and write $S$ for $\tr_nS$.

When $\C$ has finite limits, $n$-truncation $\tr_n:\sC\to\snC$ admits
a right-adjoint $\csk_n:\snC\to\sC$, called the $n$-coskeleton. The
composition
\begin{equation*}
  \csk_n\circ\tr_n : \sC \to \sC
\end{equation*}
is denoted $\Csk_n$. When $\C$ does not possess finite limits, the
functor $\Csk_n$ is only partially defined.

In the category of simplicial sets, there is also a left-adjoint
$\sk_n:\snst\to\sst$ to $n$-truncation, called the $n$-skeleton. The
composition
\begin{equation*}
  \sk_n\circ\tr_n : \sst \to \sst
\end{equation*}
is denoted $\Sk_n$.

Special cases of the next two lemmas first appeared in
\cite{DHI:04}. Let $S\hookrightarrow T$ be a monomorphism of finite
simplicial sets.
\begin{lemma}
  \label{lemma:phrep}
  Suppose that $S$ is $n$-dimensional. Let $Y_\bullet\in\sC$ be a
  simplicial object such that the limit $\hom(T,Y)$ exists. Let $f:\tr_nX_\bullet\to\tr_nY_\bullet$ be a map in
  $\snC$ such that the matching object $M_k(f)$ exists for all
  $k\leq n$, and the map
  \begin{equation*}
    \mu_k(f) : X_k \to M_k(f)
  \end{equation*}
  is a cover for all $k\le n$.  Then the limit
  \begin{equation*}
    \hom(S,X)\times_{\hom(S,Y)}\hom(T,Y)
  \end{equation*}
  exists.
\end{lemma}
\begin{proof}
  Filter the simplicial set $S$
  \begin{equation*}
    \emptyset = S_0 \hookrightarrow \ldots \hookrightarrow S_N = S
  \end{equation*}
  where
  \begin{equation*}
    S_\ell \cong S_{\ell-1} \cup_{\partial\Delta^{n_\ell}} \Delta^{n_\ell} .
  \end{equation*}
  Here, $n_\ell\leq n$ for all $\ell$.

  Suppose that the limit
  \begin{equation*}
    Z_j = \hom(S_j,X)\times_{\hom(S_j,Y)}\hom(T,Y) .
  \end{equation*}
  exists for $j<\ell$. This is true for $\ell=1$, since
  $Z_0\cong\hom(T,Y)$. The limit $Z_\ell$ is the pullback
  \begin{equation*}
    \begin{xy}
      \Square[Z_\ell`X_{n_\ell}`Z_{\ell-1}`M_{n_\ell}(f);``\mu_{n_\ell}(f)`]
    \end{xy}
  \end{equation*}
  This pullback exists because $\mu_{n_\ell}(f)$ is a cover.
\end{proof}

\begin{lemma}
  \label{lemma:hrep}
  Let $f:X_\bullet\to Y_\bullet$ be a hypercover such
  that the limit
  \begin{equation*}
    \hom(S,X)\times_{\hom(S,Y)}\hom(T,Y)
  \end{equation*}
  exists. Then the limit $\hom(T,X)$ exists
  and the map
  \begin{equation*}
    \hom(T,X) \to \hom(S,X)\times_{\hom(S,Y)}\hom(T,Y)
  \end{equation*}
  is a cover.
\end{lemma}
\begin{proof}
  Filter the simplicial set $T$
  \begin{equation*}
    S = S_0 \hookrightarrow \ldots \hookrightarrow S_N = T
  \end{equation*}
  where
  \begin{equation*}
    S_\ell \cong S_{\ell-1} \cup_{\partial\Delta^{n_\ell}} \Delta^{n_\ell} .
  \end{equation*}

  Suppose that the limit
  \begin{equation*}
    Z_j = \hom(S_j,X)\times_{\hom(S_j,Y)}\hom(T,Y) .
  \end{equation*}
  exists for $j<\ell$, and that the map $Z_j\to Z_0$ is a
  cover. The limit
  \begin{equation*}
    Z_0=\hom(S,X)\times_{\hom(S,Y)}\hom(T,Y)
  \end{equation*}
  exists by hypothesis. The limit $Z_\ell$ is the pullback
  \begin{equation*}
    \begin{xy}
      \Square[Z_\ell`X_{n_\ell}`Z_{\ell-1}`M_{n_\ell}(f);``\mu_{n_\ell}(f)`]
    \end{xy}
  \end{equation*}
  This pullback exists because $\mu_{n_\ell}(f)$ is a cover.

  We conclude that the limit $Z_N=\hom(T,X)$ exists, and that the
  morphism $Z_N\to Z_0$ is a cover.
\end{proof}

\begin{theorem}
  \label{prop:hk2}\mbox{}
  \begin{enumerate}
  \item The composition of two $n$-hypercovers is an $n$-hypercover.
  \item The pullback of an $n$-hypercover along a map of Lie
    $\infty$-groupoids exists in $\sC$, and is an $n$-hypercover.
  \end{enumerate}
\end{theorem}
\begin{proof}
  Consider a composable pair of $n$-hypercovers
  \begin{equation*}
    \begin{xy}
      \morphism<400,0>[X_\bullet`Y_\bullet;g]
      \morphism(400,0)<400,0>[Y_\bullet`Z_\bullet;f]
    \end{xy}
  \end{equation*}
  Suppose that the matching object $M_j(fg)$ exists and that the
  map $\mu_j(fg)$ is a cover for $j<k$. This is certainly the case for
  $k=1$, since $M_0(fg)\cong Z_0$, and $\mu_0(fg)=f_0g_0$ is the
  composition of the two covers $f_0$ and $g_0$. Lemma
  \ref{lemma:phrep} now shows that the matching object $M_k(fg)$ exists.  The square in the commuting diagram
  \begin{equation*}
    \begin{xy}
      \qtriangle<600,500>[X_k`M_k(g)`M_k(fg);\mu_k(g)`\mu_k(fg)`]
      \Square(600,0)[M_k(g)`Y_k`M_k(fg)`M_k(f);``\mu_k(f)`]
    \end{xy}
  \end{equation*}
  is a pullback. Since $f$ and $g$ are $n$-hypercovers, we see that
  $\mu_k(fg)$ is a (composition of) cover(s) for all $k$, and an
  isomorphism if $k>n$.

  We turn to the second statement. Consider an $n$-hypercover
  $f:X_\bullet\to Z_\bullet$ and a map $g:Y_\bullet\to
  Z_\bullet$ of Lie $\infty$-groupoids. Suppose that the limits
  $g^*X_j$ and $M_j(g^*f)$ exist and that the maps
  $\mu_j(g^*f)$ are covers for $j<k$. Lemma \ref{lemma:phrep} shows
  that the matching object $M_k(g^*f)$ exists. The limit $g^\ast X_k$ is the pullback
  \begin{equation*}
    \begin{xy}
      \Square[g^*X_k`X_k`M_k(g^*f)`M_k(f);`\mu_k(g^*f)`\mu_k(f)`]
    \end{xy}
  \end{equation*}
  The map $\mu_k(f)$ is a cover for all $k$ because
  $f$ is an $n$-hypercover. This shows that the pullback $g^*X_k$ exists, that the map $\mu_k(g^*f)$ is a cover for all $k$, and that it
  is an isomorphism for $k\ge n$.
\end{proof}

There is also a relative version of the notion of a Lie $n$-groupoid,
modeled on the definition of a Kan fibration in the theory of
simplicial sets.
\begin{definition}\label{defi:nst}
  Let $n\in\mathbb{N}\cup\{\infty\}$. A map $f:X_\bullet\to
  Y_\bullet$ of Lie $\infty$-groupoids is an \emph{$n$-stack} if for
  all $k>0$ and $0\le i\le k$, the limit $\Lambda^k_i(f)$ exists, the map
  \begin{equation*}
    \lambda^k_i(f) : X_k \to \Lambda^k_i(f)
  \end{equation*}
  is a cover, and it is an isomorphism if $k>n$.
\end{definition}

There are analogues of Lemmas \ref{lemma:phrep} and \ref{lemma:hrep}
for $n$-stacks, due to Henriques \cite{Hen:08}, but only under certain
additional conditions on the simplicial sets $S$ and $T$.
\begin{definition}
  An inclusion of finite simplicial sets $S\hookrightarrow T$ is an
  \emph{expansion} if it can be written as a composition
  \begin{equation*}
    S = S_0 \hookrightarrow \cdots \hookrightarrow S_N = T
  \end{equation*}
  where
  \begin{equation*}
    S_\ell \cong S_{\ell-1} \cup_{\Lambda^{n_\ell}_{i_\ell}} \Delta^{n_\ell} .
  \end{equation*}
  A finite simplicial set $S$ is \emph{collapsible} if the inclusion
  of some, and hence any, vertex is an expansion.
\end{definition}

Let $S\hookrightarrow T$ be a monomorphism of finite
simplicial sets.
\begin{lemma}
  \label{lemma:psrep}
  Suppose that $S$ is $n$-dimensional and collapsible. Let
  $Y_\bullet\in\sC$ be a simplicial object such that the limit
  $\hom(T,Y)$ exists and the restriction to any vertex
  \begin{equation*}
    \hom(T,Y) \to Y_0
  \end{equation*}
  is a cover. Let $f\colon\tr_nX_\bullet\to\tr_nY_\bullet$ be a map in
  $\snC$ such that the limit $\Lambda^k_i(f)$ exists
  for all $0<k\leq n$ and $0\le i\le n$, and the map
  \begin{equation*}
    \lambda^k_i(f) : X_k \to \Lambda^k_i(f)
  \end{equation*}
  is a cover for all $0<k\le n$ and $0\le i\le n$.  Then the limit
  \begin{equation*}
    \hom(S,X)\times_{\hom(S,Y)}\hom(T,Y)
  \end{equation*}
  exists.
\end{lemma}
\begin{proof}
  Filter the simplicial set $S$
  \begin{equation*}
    \Delta^0 = S_0 \hookrightarrow \ldots \hookrightarrow S_N = S
  \end{equation*}
  where
  \begin{equation*}
    S_\ell \cong S_{\ell-1} \cup_{\Lambda^{n_\ell}_{i_\ell}}
    \Delta^{n_\ell} .
  \end{equation*}
  Here, $n_\ell\leq n$ for all $\ell$.

  Suppose that the limit
  \begin{equation*}
    Z_j = \hom(S_j,X)\times_{\hom(S_j,Y)}\hom(T,Y) .
  \end{equation*}
  exists for $j<\ell$. This is true for $\ell=1$, by the
  hypotheses on $\hom(T,Y)$. We have the pullback diagram
  \begin{equation*}
    \begin{xy}
      \Square[Z_\ell`X_{n_\ell}`Z_{\ell-1}`\Lambda^{n_\ell}_{i_\ell}(f);%
      ``\lambda^{n_\ell}_{i_\ell}(f)`]
    \end{xy}
  \end{equation*}
  This pullback exists because $\lambda^{n_\ell}_{i_\ell}(f)$ is a
  cover.
\end{proof}

\begin{lemma}
  \label{lemma:srep}
  Let $f:X_\bullet\to Y_\bullet$ be an $\infty$-stack such that
  the limit
  \begin{equation*}
    \hom(S,X)\times_{\hom(S,Y)}\hom(T,Y)
  \end{equation*}
  exists. Suppose that the inclusion $S\hookrightarrow T$ is
  an expansion. Then the limit $\hom(T,X)$ exists, and
  the map
  \begin{equation*}
    \hom(T,X) \to \hom(S,X)\times_{\hom(S,Y)}\hom(T,Y)
  \end{equation*}
  is a cover.
\end{lemma}
\begin{proof}
  Filter the simplicial set $T$
  \begin{equation*}
    S = S_0 \hookrightarrow \ldots \hookrightarrow S_N = T
  \end{equation*}
  where
  \begin{equation*}
    S_\ell \cong S_{\ell-1} \cup_{\Lambda^{n_\ell}_{i_\ell}} \Delta^{n_\ell} .
  \end{equation*}

  Suppose that the limit
  \begin{equation*}
    Z_j = \hom(S_j,X)\times_{\hom(S_j,Y)}\hom(T,Y) .
  \end{equation*}
  exists for $j<\ell$, and the map $Z_j\to Z_0$ is a
  cover. The limit
  \begin{equation*}
    Z_0=\hom(S,X)\times_{\hom(S,Y)}\hom(T,Y)
  \end{equation*}
  exists by hypothesis. We have the pullback diagram
  \begin{equation*}
    \begin{xy}
      \Square[Z_\ell`X_{n_\ell}`Z_{\ell-1}`\Lambda^{n_\ell}_{i_\ell}(f);%
      ``\lambda^{n_\ell}_{i_\ell}(f)`]
    \end{xy}
  \end{equation*}
  This pullback exists because $\lambda^{n_\ell}_{i_\ell}(f)$
  is a cover.

  We conclude that the limit $Z_N=\hom(T,X)$ exists, and that the
  morphism $Z_N\to Z_0$ is a cover.
\end{proof}

\begin{theorem}
  \label{thm:hk1} \mbox{}
  \begin{enumerate}
  \item An $n$-hypercover is an $n$-stack.
  \item A hypercover which is an $n$-stack is an $n$-hypercover.
  \item The composition of two $n$-stacks is an $n$-stack.
  \item Let $f\colon X_\bullet\to Y_\bullet$ be an $n$-stack and let
    $g\colon Z_\bullet\to Y_\bullet$ be a map of Lie $\infty$-groupoids.
    If the pullback of $f_0$ along $g_0$ exists in $\C$, then the
    pullback of $f$ along $g$ exists in $\sC$ and this pullback is an
    $n$-stack.
  \end{enumerate}
\end{theorem}
\begin{proof}
  Let $f\colon X_\bullet\to Y_\bullet$ be an $\infty$-hypercover. We see
  that $f$ is an $\infty$-stack by considering the finite inclusions
  $\Lambda^k_i\hookrightarrow\Delta^k$ and applying Lemma
  \ref{lemma:hrep}. It remains to show that if $f$ is an
  $n$-hypercover, then $\lambda^k_i(f)$ is an isomorphism when $k>n$ (and
  $0\le i\le k$).

  The square in the commuting diagram
  \begin{equation}\label{mulambda}
    \begin{xy}
      \qtriangle(0,-250)<600,500>[X_k`M_k(f)`\Lambda^k_i(f);\mu_k(f)`\lambda^k_i(f)`]
      \Square(600,-250)[M_k(f)`X_{k-1}`\Lambda^k_i(f)`M_{k-1}(f);``\mu_{k-1}(f)`]
    \end{xy}
  \end{equation}
  is a pullback. If $f$ is an $n$-hypercover and $k>n$, then the maps
  $\mu_k(f)$ and $\mu_{k-1}(f)$ are isomorphisms, and we see that
  $\lambda^k_i(f)$ is an isomorphism.

  To prove the second part, we consider \eqref{mulambda} in the case
  where $f$ is a hypercover and an $n$-stack. If $k>n$, then $\lambda^k_i(f)$ is an isomorphism.  The map
  $\mu_k(f)$ is a cover, and, by Axiom \ref{axiom:subcan}, an epimorphism.  The diagram \eqref{mulambda} now implies that $\mu_k(f)$ is an
  isomorphism.  Similarly, the map $M_k(f)\to\Lambda^k_i(f)$ is an isomorphism.

  The map $\Lambda^k_i(f)\to M_{k-1}(f)$ is induced by the inclusion $\partial\Delta^{k-1}\hookrightarrow\Lambda^k_i$. Lemma \ref{lemma:hrep} shows that it is a cover.  The pull-back of $\mu_{k-1}(f)$ along this cover is the map $M_k(f)\to\Lambda^k_i(f)$.  This map is an isomorphism for $k>n$.  Axiom \ref{axiom:subcan} therefore implies that $\mu_{k-1}(f)$ is an isomorphism for $k>n$. Hence $f$ is an $n$-hypercover.

  Turning to the third part of the theorem, consider a composable pair
  of $n$-stacks
  \begin{equation*}
    \begin{xy}
      \morphism<400,0>[X_\bullet`Y_\bullet;g]
      \morphism(400,0)<400,0>[Y_\bullet`Z_\bullet;f]
    \end{xy}
  \end{equation*}
  Suppose that the limit $\Lambda^j_i(fg)$ exists and
  that the map $\lambda^j_i(fg)$ is a cover, for all $0<j<k$ (and
  $0\le i\le j$). Lemma \ref{lemma:srep} now shows that the limit
  $\Lambda^k_i(fg)$ exists.  The square in the commuting
  diagram
  \begin{equation*}
    \begin{xy}
      \qtriangle<600,500>[X_k`\Lambda^k_i(g)`\Lambda^k_i(fg);%
      \lambda^k_i(g)`\lambda^k_i(fg)`]
      \Square(600,0)[\Lambda^k_i(g)`Y_k`\Lambda^k_i(fg)`\Lambda^k_i(f);%
      ``\lambda^k_i(f)`]
    \end{xy}
  \end{equation*}
  is a pullback. Since $f$ and $g$ are $n$-stacks, we see that
  $\lambda^k_i(fg)$ is a (composition of) cover(s) for all $k>0$ and
  $0\le i\le k$, and an isomorphism if $k>n$.

  We turn to the fourth statement. Consider an $n$-stack
  $f:X_\bullet\to Z_\bullet$ and a map $g:Y_\bullet\to Z_\bullet$ of
  Lie $\infty$-groupoids. Suppose that, for $j<k$ (and $0\le i\le j$), the pullback $g^*X_j$ exists, the limit $\Lambda^j_i(g^*f)$ exists, and the map $\lambda^j_i(g^*f)$ is a cover. Lemma \ref{lemma:psrep} shows that the limit
  $\Lambda^k_i(g^*f)$ exists for $0\le i\le k$. The limit $g^\ast X_k$ is the pullback
  \begin{equation*}
    \begin{xy}
      \Square[g^*X_k`X_k`\Lambda^k_i(g^*f)`\Lambda^k_i(f);%
      `\lambda^k_i(g^*f)`\lambda^k_i(f)`]
    \end{xy}
  \end{equation*}
  The map $\lambda^k_i(f)$ is a cover for all $k>0$ because $f$ is an $n$-hypercover. This shows that the pull-back $g^*X_k$ exists, that the map $\lambda^k_i(g^*f)$ is a cover for all $k>0$ and $0\le i\le k$, and that this map is an isomorphism for
  $k>n$.
\end{proof}

\section{Higher Morphism Spaces in Higher Stacks}
\label{sec:highhom}

Order preserving maps of all finite ordinals, including 0, form a
category $\Delta_+$ extending $\Delta$.  An augmented simplicial set
is a functor
\begin{equation*}
  \begin{xy}
    \morphism[\Delta_+^{\circ}`\st;]
  \end{xy}
\end{equation*}
Such a functor consists of a simplicial set $S$ equipped with a map to
a constant simplicial set $S_{-1}$.

The ordinal sum
\begin{equation*}
  [n]+[m]:=\{0\leq\ldots\leq n\leq 0'\leq\ldots\leq m'\}=[n+m+1]
\end{equation*}
endows the category $\Delta_+$ with a monoidal structure.  This
structure extends along the Yoneda embedding
\begin{equation*}
  \begin{xy}
    \morphism[\Delta_+`\sst_+;]
  \end{xy}
\end{equation*}
to give a closed monoidal structure on $\sst_+$ called the
\emph{join}, and denoted $\star$. Given an augmented simplicial set
$K$, denote the right adjoint to $K\star(-)$ by
\begin{equation*}
  \begin{xy}
    \morphism<750,0>[\sst_+`\sst_+;(-)^{K\star}]
  \end{xy}
\end{equation*}

\begin{example}
  The functor $(-)^{\Delta^{n-1}\star}$ is Illusie's $\Dec_n(-)$
  (c.f.\ \cite[Chapter VI]{Ill:72}).
\end{example}

The inclusion $i:\Delta\hookrightarrow\Delta_+$ provides a forgetful
functor
\begin{equation*}
  \begin{xy}
    \morphism[\sst_+`\sst;i^*]
  \end{xy}
\end{equation*}
Its right adjoint
\begin{equation*}
  \begin{xy}
    \morphism[\sst`\sst_+;i_*]
  \end{xy}
\end{equation*}
augments a simplicial set by a point.

\begin{definition}\mbox{}
  \begin{enumerate}
  \item Let $S$ and $T$ be simplicial sets. Denote by $S\star T$ the
    simplicial set
    \begin{equation*}
      S\star T = i^*((i_*S)\star(i_*T)) .
    \end{equation*}
  \item Let $X_{\bullet}$ be in $\sC$ and let $S$ be a finite
    simplicial set.  Denote by $X^{S\star}_{\bullet}$ the putative simplicial object with $k$-simplices
    \begin{align*}
      X^{S\star}_k:=\hom(S\star\Delta^k,X)
    \end{align*}
    Face and degeneracy maps are given by $1\star d_i$ and $1\star
    s_i$.
  \end{enumerate}
\end{definition}

In \cite{Dus:02}, Duskin gave a construction of the collection of
morphisms in a higher category.\footnote{Duskin called this the
  ``path-homotopy complex''.  His construction is hinted at in the
  earlier treatment of nerves in \cite{Ill:72}, and of weak Kan
  complexes in \cite{BoV:73}.}
\begin{definition}
  Let $X_{\bullet}$ be an $\infty$-groupoid in $\st$. Define
  $P^{\geq1}X_{\bullet}$ to be the pullback
  \begin{equation*}
    \begin{xy}
      \Square[P^{\geq1}X_{\bullet}`X^{\Delta^0\star}_\bullet`X_0`X_{\bullet};```]
    \end{xy}
  \end{equation*}
\end{definition}
The simplicial set $P^{\geq 1}X$ is an $\infty$-groupoid which models
the ``space'' of $1$-morphisms in $X_{\bullet}$.\footnote{Lurie
  \cite{Lur:09} denotes this construction `$\hom_X^L$'. Given $x,y\in
  X_0$, Lurie's $\hom_X^L(x,y)$ is the fiber at $(x,y)$ of a canonical
  map $P^{\geq 1}X\rightarrow\hom(\partial\Delta^1,X)$.}

We will generalize $P^{\ge1}X_\bullet$ to higher morphism spaces
$P^{\ge k}X_\bullet$, for all $k>0$, and at the same time, we will define a relative
version of the construction. The discussion follows the lines of
Joyal's proof of \cite[Theorem~3.8]{Joy:02}, except that we restrict
attention to the case
\begin{equation*}
  (S\hookrightarrow T) =
  (\partial\Delta^{k-1}\hookrightarrow\Delta^{k-1}) ,
\end{equation*}
and work in the Lie setting.

Given a proper, non-empty subset $J$ of $[n]$, let
\begin{equation*}
  \Lambda^n_J = \bigcup_{i\in J} \partial_i\Delta^n \subset \partial\Delta^n .
\end{equation*}
An induction shows that $\Lambda^n_J$ is collapsible.

Let $f:X_\bullet\to Y_\bullet$ be an $\infty$-stack. The
$\ell$-simplices of the simplicial object
\begin{equation}
  \label{star}
  X_\bullet^{\partial\Delta^{k-1}\star}\times_{Y_\bullet^{\partial\Delta^{k-1}\star}}
  Y_\bullet^{\Delta^{k-1}\star}
\end{equation}
are given by the limit
\begin{multline*}
  \hom(\partial\Delta^{k-1}\star\Delta^\ell,X)
  \times_{\hom(\partial\Delta^{k-1}\star\Delta^\ell,Y)}Y_{k+\ell} \\
  \cong \hom(\Lambda^{k+\ell}_{\{0,\ldots,k-1\}},X)
  \times_{\hom(\Lambda^{k+\ell}_{\{0,\ldots,k-1\}},Y)} Y_{k+\ell} .
\end{multline*}
Lemma \ref{lemma:psrep} implies that this limit exists. In the special case of vertices $\ell=0$, we obtain a
natural identification with the space of relative horns
$\Lambda^k_k(f)$.

Denote by $f^{\partial\Delta^{k-1}\star}$ the induced map
\begin{equation*}
  f^{\partial\Delta^{k-1}\star}\colon X^{\Delta^{k-1}\star} \to
  X_\bullet^{\partial\Delta^{k-1}\star}\times_{Y_\bullet^{\partial\Delta^{k-1}\star}}
  Y_\bullet^{\Delta^{k-1}\star} .
\end{equation*}

\begin{lemma}
  \label{star-l}
  Let $f:X_\bullet\to Y_\bullet$ be an $\infty$-stack. For $\ell>0$,
  the map $\Lambda^\ell_i(f^{\partial\Delta^{k-1}\star})$ is
  canonically isomorphic to $\Lambda^{k+\ell}_{k+i}(f)$.
\end{lemma}
\begin{proof}
  An exercise in the combinatorics of joins shows that
  \begin{align*}
    ( \Delta^{k-1}\star\Lambda^\ell_i \hooklongrightarrow
    \Delta^{k-1}\star\Delta^\ell ) & \cong (
    \Lambda^{k+\ell}_{\{k,\ldots,\widehat{k+i},\ldots,k+\ell\}}
    \hooklongrightarrow \Delta^{k+\ell} ) , \text{ and} \\
    ( \partial\Delta^{k-1}\star\Delta^\ell \hooklongrightarrow
    \Delta^{k-1}\star\Delta^\ell ) & \cong (
    \Lambda^{k+\ell}_{\{0,\ldots,k-1\}} \hooklongrightarrow
    \Delta^{k+\ell} ) .
  \end{align*}
  In this way, we obtain a pushout square
  \begin{equation*}
    \begin{xy}
      \square<1000,500>[\partial\Delta^{k-1}\star\Lambda^\ell_i`
      \partial\Delta^{k-1}\star\Delta^\ell`
      \Delta^{k-1}\star\Lambda^\ell_i`\Lambda^{k+\ell}_{k+i};```]
    \end{xy}
  \end{equation*}
  which gives rise to the pullback square
  \begin{equation*}
    \begin{xy}
      \square<1550,500>[\Lambda^{k+\ell}_{k+i}(f)`%
      (X^{\partial\Delta^{k-1}\star}\times_{Y^{\partial\Delta^{k-1}\star}}Y^{\Delta^{k-1}\star})_{\ell}`
      \Lambda^\ell_i(X^{\Delta^{k-1}\star})`%
      \Lambda^{\ell}_i(X^{\partial\Delta^{k-1}\star}\times_{Y^{\partial\Delta^{k-1}\star}}Y^{\Delta^{k-1}\star});```]
    \end{xy}
  \end{equation*}
  But this is also the pullback square defining
  $\Lambda^\ell_i(f^{\partial\Delta^{k-1}\star})$.
\end{proof}

It follows that if $f$ is an $\infty$-stack, then so is
$f^{\partial\Delta^{k-1}\star}$, and if $f$ is an $n$-stack,
$f^{\partial\Delta^{k-1}\star}$ is an $(n-k)$-stack.

\begin{definition}\label{defi:relkmor}
  For $k>0$, the \emph{relative higher morphism space} $P^{\geq
    k}(f)_{\bullet}$ is the pullback
  \begin{equation*}
    \begin{xy}
      \Square[P^{\geq k}(f)_{\bullet}`X^{\Delta^{k-1}\star}_{\bullet}`\Lambda^k_k(f)`%
      X_\bullet^{\partial\Delta^{k-1}\star}\times_{Y_\bullet^{\partial\Delta^{k-1}\star}}%
      Y_\bullet^{\Delta^{k-1}\star};```]
    \end{xy}
  \end{equation*}
\end{definition}

The following result is an immediate consequence of Lemma~\ref{star-l}
and Theorem~\ref{thm:hk1}.
\begin{theorem}\label{thm:kmor}
  If $f:X_{\bullet}\to Y_{\bullet}$ is an $n$-stack, the relative
  higher morphism space $P^{\geq k}(f)_{\bullet}$ is a Lie
  $(n-k)$-groupoid.
\end{theorem}

The vertices of $P^{\geq k}(f)_{\bullet}$ are the $k$-simplices
of $X_{\bullet}$. Its $1$-simplices correspond to
$(k+1)$-simplices $x\in X_{k+1}$ such that
\begin{align*}
  f(x) &= s_kd_kf(x) \intertext{and, for $i<k$,}
  d_ix &= s_{k-1}d_{k-1}d_ix .
\end{align*}
We interpret $x$ as a path rel boundary in the fiber of $f$, which
begins at $d_{k+1}x$ and ends at $d_kx$.

When the matching object $M_k(f)$ exists, it provides a natural augmentation
\begin{equation}
  \label{augmentation}
  \pi\colon P^{\geq k}(f)_{\bullet} \to M_k(f) .
\end{equation}
The map underlying $\pi\colon P^{\geq k}(f)_0\cong X_k \to M_k(f)$ is
$\mu_k(f)$. This determines an augmentation because the diagram
\begin{equation*}
  \begin{xy}
    \morphism|a|/{@{>}@<3pt>}/<700,0>[P^{\geq k}(f)_1`P^{\geq k}(f)_0;\p_0]
    \morphism|b|/{@{>}@<-3pt>}/<700,0>[P^{\geq k}(f)_1`P^{\geq k}(f)_0;\p_1]
    \morphism(700,0)<600,0>[P^{\geq k}(f)_0`M_k(f);\pi]
  \end{xy}
\end{equation*}
commutes.

The augmentation \eqref{augmentation} encodes the idea that vertices in the relative higher morphism space
are lifts of a $k$-morphism in $Y_{\bullet}$, that edges are paths rel
boundary between these lifts such that the paths live in the fiber of
$f$, etc.

\begin{theorem}
  \label{thm:kmorhyp}
  If $f$ is an $n$-hypercover, then the augmentation
  \eqref{augmentation} exists and is an $(n-k)$-hypercover.
\end{theorem}
\begin{proof}
  We prove that $\pi$ is an $(n-k)$-hypercover by showing that the
  square
  \begin{equation*}
    \begin{xy}
      \square<1000,500>[P^{\geq k}(f)_\ell`X_{k+\ell}`M_\ell(\pi)`M_{k+\ell}(f);%
      `\mu_\ell(\pi)`\mu_{k+\ell}(f)`]
    \end{xy}
  \end{equation*}
  is a pullback. For $\ell=0$, this is evident: both horizontal maps
  are isomorphisms in this case.

  For $\ell>0$, the above square is the top half of a commutative
  diagram
  \begin{equation*}
    \begin{xy}
      \square<1200,500>[P^{\geq k}(f)_\ell`X_{k+\ell}`M_\ell(\pi)`M_{k+\ell}(f);%
      `\mu_\ell(\pi)`\mu_{k+\ell}(f)`]
      \square(0,-500)<1200,500>[M_\ell(\pi)`M_{k+\ell}(f)`
      \Lambda^k_k(f)`(X^{\partial\Delta^{k-1}\star}
      \times_{Y^{\partial\Delta^{k-1}\star}}Y^{\Delta^{k-1}\star})_\ell;```]
    \end{xy}
  \end{equation*}
  The outer rectangle is a pullback by definition. Thus, it suffices
  to prove that the bottom square is a pullback.

  For $\ell=1$, this may be checked directly. For $\ell>1$, an exercise in the combinatorics of joins shows that
  we have a pushout square
  \begin{equation*}
    \begin{xy}
      \square<1000,500>[\partial\Delta^{k-1}\star\partial\Delta^\ell`
      \partial\Delta^{k-1}\star\Delta^\ell`
      \Delta^{k-1}\star\partial\Delta^\ell`\partial\Delta^{k+\ell};```]
    \end{xy}
  \end{equation*}
  which gives rise to a pullback square
  \begin{equation*}
    \begin{xy}
        \square<1550,500>[M_{k+\ell}(f)`%
        M_\ell(X^{\Delta^{k-1}\star})`%
        (X^{\partial\Delta^{k-1}\star}\times_{Y^{\partial\Delta^{k-1}\star}}Y^{\Delta^{k-1}\star})_\ell`
        M_{\ell}(X^{\partial\Delta^{k-1}\star}\times_{Y^{\partial\Delta^{k-1}\star}}Y^{\Delta^{k-1}\star});```]
    \end{xy}
  \end{equation*}
  This square embeds in a commutative diagram
  \begin{equation*}
    \begin{xy}
        \square<1000,500>[M_\ell(P^{\geq k}(f))`M_{k+\ell}(f)`\Lambda^k_k(f)`(X^{\partial\Delta^{k-1}\star}\times_{Y^{\partial\Delta^{k-1}\star}}Y^{\Delta^{k-1}\star})_\ell;```]
        \square(1000,0)<1550,500>[M_{k+\ell}(f)`%
        M_\ell(X^{\Delta^{k-1}\star})`%
        (X^{\partial\Delta^{k-1}\star}\times_{Y^{\partial\Delta^{k-1}\star}}Y^{\Delta^{k-1}\star})_\ell`
        M_{\ell}(X^{\partial\Delta^{k-1}\star}\times_{Y^{\partial\Delta^{k-1}\star}}Y^{\Delta^{k-1}\star});```]
    \end{xy}
  \end{equation*}
  The outer rectangle is a pullback because $M_\ell(-)$ commutes with limits.  We observed above that the right square is a pullback. We conclude that the left square is a pullback, completing the proof that $\mu_\ell(\pi)$ is a cover.
\end{proof}

\section{Strictification}\label{sec:stric}
In this section we recall Duskin's ``$n$-strictification'' functor
$\tn$ for $n\ge 0$. This is a partially defined left-adjoint to the
inclusion of the category of $n$-stacks into the category of
$\infty$-stacks.  We establish its main properties in Propositions
\ref{prop:tn} and \ref{prop:tnhyp}.

Let $f\colon X_{\bullet}\to Y_{\bullet}$ be an $\infty$-stack such
that the orbit space $\pi_0(P^{\geq n}(f))$ exists.  We define
a map
\begin{equation}\label{eq:ntn}
  \tr_n\tn(f)\colon\tr_n\tn(X,f)_\bullet\to\tr_n Y_\bullet
\end{equation}
On $k$-simplices, for $k<n$, $\tr_n\tn(f)$ is the map
\begin{align*}
  f_k\colon X_k\to Y_k
\end{align*}
On $n$-simplices, $\tr_n\tn(f)$ is the canonical map
\begin{align*}
  \pi_0(P^{\geq n}(f))\to Y_n
\end{align*}

\begin{lemma}\label{lemma:di1}
  Let $f\colon X_{\bullet}\to Y_{\bullet}$ be an $\infty$-stack such
  that the orbit space $\pi_0(P^{\geq n}(f))$ exists.  The maps $\lambda^k_i(\tn(f))$ are covers for
  $k\le n$. For all $i$, the limit $\Lambda^{n+1}_i(\tn(f))$ exists and the map $\Lambda^{n+1}_i(f)\to\Lambda^{n+1}_i(\tn(f))$ is a
  cover.
\end{lemma}
\begin{proof}
  For $k<n$, the natural map $\tr_n
  X_\bullet\to\tr_n\tn(X,f)_\bullet$ induces an isomorphism between
  the maps $\lambda^k_i(f)$ and $\lambda^k_i(\tr_n\tn(f))$.  For
  $k=n$, we have a commuting square
  \begin{equation*}
    \begin{xy}
      \square<750,500>[X_n`\tn(X,f)_n`\Lambda^n_i(f)`\Lambda^n_i(\tn(f));%
      `\lambda^n_i(f)`\lambda^n_i(\tn(f))`\cong]
    \end{xy}
  \end{equation*}
  This square guarantees that the map $\lambda^n_i(\tn(f))$ is a cover
  for all $i$.  Indeed, the top horizontal map is the cover
  $X_n\to\pi_0(P^{\geq n}(f))$, the map $\lambda^n_i(f)$ is a cover by
  assumption, and the bottom horizontal map is an isomorphism.
  Axiom~\ref{axiom:fg+g} implies that $\lambda^n_i(\tn(f))$ is a
  cover.

  Lemma \ref{lemma:psrep} guarantees that, for all $i$, the limit $\Lambda^{n+1}_i(\tn(f))$ exists.  It remains to show that the map $\Lambda^{n+1}_i(f)\to\Lambda^{n+1}_i(\tn(f))$ is a cover for all $i$.

  Recall that given a proper, non-empty subset $J$ of $[n+1]$,
  \begin{equation*}
    \Lambda^{n+1}_J=\bigcup_{i\in J} \partial_i\Delta^{n+1}
    \subset \partial\Delta^{n+1} .
  \end{equation*}
  For each $J$, there is a map
  \begin{multline*}
    \Lambda^{n+1}_J(f) =
    \hom(\Lambda^{n+1}_J,X)\times_{\hom(\Lambda^{n+1}_J,Y)}Y_{n+1} \\
    \too \Lambda^{n+1}_J(\tn(f)) =
    \hom(\Lambda^{n+1}_J,\tn(X,f))\times_{\hom(\Lambda^{n+1}_J,Y)}Y_{n+1} .
  \end{multline*}
  We will show that it is a cover, by induction on $|J|$.

  Let $J_+=J\cup\{j\}$, where $j\notin J$, and let
  \begin{equation}\label{eq:Jintj}
    \Lambda^{n+1}_{J\cap j}(f) = \hom(\Lambda^{n+1}_J\cap\partial_j\Delta^{n+1},X)
    \times_{\hom(\Lambda^{n+1}_J\cap\partial_j\Delta^{n+1},Y)} Y_{n+1} .
  \end{equation}
  We have a pair of pullback diagrams in which the vertical maps are
  covers:
  \begin{equation*}
    \begin{xy}
      \Square[\Lambda^{n+1}_J(f)\times_{\Lambda^{n+1}_{J\cap j}(f)}\pi_0(P^{\le n}(f))`%
      \Lambda^{n+1}_J(f)`\Lambda^{n+1}_{J_+}(\tn(f))`%
      \Lambda^{n+1}_J(\tn(f));```]
    \end{xy}
  \end{equation*}
  and
  \begin{equation*}
    \begin{xy}
      \Square[\Lambda^{n+1}_{J_+}(f)`X_n`%
      \Lambda^{n+1}_J(f)\times_{\Lambda^{n+1}_{J\cap j}(f)}\pi_0(P^{\le n}(f))`%
      \pi_0(P^{\le n}(f));```]
    \end{xy}
  \end{equation*}
  Composing the left vertical arrows, we see that the map
  \begin{equation*}
    \Lambda^{n+1}_{J_+}(f) \too \Lambda^{n+1}_{J_+}(\tn(f))
  \end{equation*}
  is a cover; this completes the induction step.
\end{proof}

Our goal is now to construct, for any $i$, a ``missing face map''
\begin{equation*}
  d_i\colon\Lambda^{n+1}_i(\tn(f))\to\pi_0(P^{\ge n}(f))=\tn(X,f)_n .
\end{equation*}

Compose the covers $\Lambda^{n+1}_i(f)\to\Lambda^{n+1}_i(\tn(f))$ and
$\lambda^{n+1}_i(f)$ to obtain a cover
$X_{n+1}\to\Lambda^{n+1}_i(\tn(f))$.  Denote by $q d_i$ the composite
\begin{equation*}
    \begin{xy}
        \morphism[X_{n+1}`X_n;d_i]
        \morphism(500,0)[X_n`\pi_0(P^{\ge n}(f));]
    \end{xy}
\end{equation*}

\begin{lemma}\label{lemma:di2}
  The diagram
  \begin{equation}\label{eq:di}
    \begin{xy}
      \morphism/{@{>}@<3pt>}/<1000,0>[X_{n+1}\times_{\Lambda^{n+1}_i(\tn(f))}X_{n+1}`X_{n+1};]
      \morphism/{@{>}@<-3pt>}/<1000,0>[X_{n+1}\times_{\Lambda^{n+1}_i(\tn(f))}X_{n+1}`X_{n+1};]
      \morphism(1000,0)<750,0>[X_{n+1}`\pi_0(P^{\geq n}(f));qd_i]
    \end{xy}
  \end{equation}
  commutes.
\end{lemma}

\begin{remark}
  Recall that a \emph{point} of $\C$ is a functor $p\colon\C\to\st$
  which preserves finite limits, which preserves arbitrary colimits,
  and which takes covers to surjections.  We say $\C$ has \emph{enough
    points} if for every pair of maps $f\neq g$ in $\C$, there exists
  a point $p$ such that $p(f)\neq p(g)$.

  We prove the lemma under the assumption that $\C$ has enough points.
  This assumption is satisfied in many examples of interest, and has
  the benefit of allowing for an elementary proof.

  One could proceed without this assumption by using Ehresman's theory of
  sketches in combination with Barr's theorem on the existence of a
  Boolean cover of the topos of sheaves on $\C$.  The latter approach
  is discussed in \cite[Section 2 -- ``For Logical Reasons'']{Bek:04}
  or in more depth in \cite[Chapter 7, especially 7.5]{Joh:77}.
\end{remark}

\begin{proof}
  To show that \ref{eq:di} commutes, we construct an epi
  \begin{equation*}
        \mathbb{K}\to^g X_{n+1}\times_{\Lambda^{n+1}_i(\tn(f))}X_{n+1}
  \end{equation*}
  and we show that the diagram
  \begin{equation}\label{eq:difat}
    \begin{xy}
        \morphism/{@{>}@<3pt>}/[\mathbb{K}`X_{n+1};]
        \morphism/{@{>}@<-3pt>}/[\mathbb{K}`X_{n+1};]
        \morphism(500,0)<750,0>[X_{n+1}`\pi_0(P^{\geq n}(f));qd_i]
    \end{xy}
  \end{equation}
  commutes. This implies that \ref{eq:di} commutes.

  {\bf Step 1: Construct $g$.}

  By way of motivation, the fork
  \begin{equation*}
    \begin{xy}
      \morphism/{@<3pt>}/[\mathbb{K}`X_{n+1};]
      \morphism/{@<-3pt>}/[\mathbb{K}`X_{n+1};]
    \end{xy}
  \end{equation*}
  can be understood as a ``fattened version'' of the kernel pair of the cover $X_{n+1}\to\Lambda^{n+1}_i(\tn(f))$.

  In more detail, sections of $X_{n+1}\times_{\Lambda^{n+1}_i(\tn(f))} X_{n+1}$ consist of pairs of $n+1$-simplices of $X$ which live over the same $n+1$-simplex of $Y$ and such that their $i^{th}$-horns are homotopic rel boundary over $Y$. We will construct $\mathbb{K}$ by adding in the data of the homotopies. We start by defining a space of homotopies $\mathbb{H}^{n+1}_i(f)$ as the pullback
  \begin{equation*}
    \begin{xy}
      \morphism(0,0)<1750,0>[\mathbb{H}^{n+1}_i(f)`(P^{\ge
        n}(f)_1)^{\times(n+1)};]
      \morphism(0,0)<0,-1000>[\mathbb{H}^{n+1}_i(f)`\Lambda^{n+1}_i(f);(d_0)^{\times(n+1)}]
      \morphism(1750,0)|r|<0,-500>[(P^{\ge
        n}(f)_1)^{\times(n+1)}`(X_n)^{\times(n+1)};(d_0)^{\times(n+1)}]
      \morphism(1750,-500)<0,-500>[(X_n)^{\times(n+1)}`(\pi_0(P^{\ge
        n}(f)))^{\times(n+1)};]
      \morphism(0,-1000)<750,0>[\Lambda^{n+1}_i(f)`(X_n)^{\times(n+1)};]
      \morphism(750,-1000)<1000,0>[(X_n)^{\times(n+1)}`(\pi_0(P^{\ge
        n}(f)))^{\times(n+1)};]
    \end{xy}
  \end{equation*}
  Observe that the pullback exists because the right vertical maps are both covers. Observe that a section of $\mathbb{H}^{n+1}_i(f)$ consists of a pair of relative $i^{th}$ $(n+1)$-horns of $f$ along with choices of homotopies rel boundary over $Y$ between the $j^{th}$-faces of each relative horn for $j\neq i$.

  We now glue this data in to form $\mathbb{K}$. For this, note that in addition to the map $(d_0)^{\times(n+1)}\colon\mathbb{H}^{n+1}_i(f)\to\Lambda^{n+1}_i(f)$, there is another map $(d_1)^{\times(n+1)}\colon\mathbb{H}^{n+1}_i(f)\to\Lambda^{n+1}_i(f)$. Define $\mathbb{K}$ to be the iterated pullback
  \begin{equation*}
    \begin{xy}
      \morphism(0,1000)<0,-500>[\mathbb{K}`X_{n+1}\times_{\Lambda^{n+1}_i(f)}\mathbb{H}^{n+1}_i(f);]
      \morphism(0,1000)<1850,0>[\mathbb{K}`X_{n+1};]
      \morphism(1850,1000)|r|<0,-500>[X_{n+1}`\Lambda^{n+1}_i(f);\lambda^{n+1}_i(f)]
      \morphism(1000,500)<850,0>[\mathbb{H}^{n+1}_i(f)`\Lambda^{n+1}_i(f);(d_1)^{\times(n+1)}]
      \square<1000,500>[X_{n+1}\times_{\Lambda^{n+1}_i(f)}\mathbb{H}^{n+1}_i(f)`\mathbb{H}^{n+1}_i(f)`X_{n+1}`\Lambda^{n+1}_i(f);``(d_0)^{\times(n+1)}`\lambda^{n+1}_i(f)]
    \end{xy}
  \end{equation*}
  The limit $\mathbb{K}$ exists because $\lambda^{n+1}_i(f)$ is a cover ($f$ is an $\infty$-stack).  Observe that $\mathbb{K}$ encodes the data of pairs $(x_0,x_1)$ of $(n+1)$-simplices in the same fiber of $f$ and explicit fiber homotopies rel boundary $(p_j)_{0\le j\ne i}^{n+1}$ between all but their $i^{th}$ faces. The projections along the left and right $X_{n+1}$ factors induce a map
  \begin{equation*}
    \begin{xy}
      \morphism<1000,0>[\mathbb{K}`X_{n+1}\times_{\Lambda^{n+1}_i(\tn(f))}X_{n+1};g]
    \end{xy}
  \end{equation*}
  {\bf Step 2: Prove that $g$ is an epi.}

  Because we assume that $\C$ has enough points, it suffices to check that the map $p(g)$ is a surjection for any point $p\colon\C\to\st$.  A point $p$ preserves finite limits and arbitrary colimits, and takes covers to surjections.  As a result, $p$ takes each construction we have been considering to its analogue in the category $\st$.

  It suffices to show that if $f\colon X_{\bullet}\to Y_{\bullet}$ is a Kan fibration of simplicial sets, then the map
  \begin{equation*}
    \begin{xy}
        \morphism<1850,0>[\mathbb{K}= X_{n+1}\times_{\Lambda^{n+1}_i(f)}\mathbb{H}^{n+1}_i(f)\times_{\Lambda^{n+1}_1(f)}X_{n+1}`X_{n+1}\times_{\Lambda^{n+1}_i(\tn(f))}X_{n+1};g]
    \end{xy}
  \end{equation*}
  is surjective.  This map fits into a pullback square
  \begin{equation*}
    \begin{xy}
        \square<1500,500>[\mathbb{K}`(P^{\ge n}(f)_1)^{\times(n+1)}`X_{n+1}\times_{\Lambda^{n+1}_i(\tn(f))}X_{n+1}`(X_n\times_{\pi_0 (P^{\ge n}(f))} X_n)^{\times(n+1)};`g``]
    \end{xy}
  \end{equation*}
  The map
  \begin{equation*}
    P^{\ge n}(f)_1\to X_n\times_{\pi_0 (P^{\ge n}(f))} X_n
  \end{equation*}
  is surjective because the simplicial set $P^{\ge n}(f)_\bullet$ is Kan. As a result, the map $g$ is surjective, because surjections of sets are preserved under products and pullbacks.

  {\bf Step 3: Show that the square \eqref{eq:difat} commutes.}

  We will construct a sequence of covers
  \begin{equation}\label{eq:diseq}
    K_{n+3}\to\cdots\to K_0=\mathbb{K}
  \end{equation}
  fitting into a commuting square
  \begin{equation}\label{eq:diK}
    \begin{xy}
      \square(0,-250)<750,500>[K_{n+3}`P^{\ge n}(f)_1`\mathbb{K}`X_n\times X_n;h`c`(d_0,d_1)`(d_i,d_i)g]
    \end{xy}
  \end{equation}
  By way of motivation, we note that the restriction of this sequence of covers to a given section $(x_0,x_1,(p_j)_{0\le j\ne i}^{n+1})$ of $\mathbb{K}$ will encode the process of homotoping the $i^{th}$ horn of $x_1$ to the $i^{th}$-horn of $x_0$ one face at a time, using the specified homotopies $p_j$, and then using this to produce a homotopy between the $i^{th}$ faces of $x_0$ and $x_1$.

  Granting the existence of the sequence \eqref{eq:diseq} and the square \eqref{eq:diK}, we conclude the lemma as follows. Recall that $q$ denotes the map $X_n\too\pi_0(P^{\ge n}(f))$.  By definition,
  \begin{equation*}
    qd_0=qd_1\colon P^{\ge n}(f)_1\too\pi_0(P^{\ge n}(f))
  \end{equation*}
  The square \ref{eq:diK} implies that
  \begin{align*}
    qd_i\pr_1gc&=qd_0h\\
    &=qd_1h\\
    &=qd_i\pr_2gc
  \end{align*}
  The map $c\colon K_{n+3}\too\mathbb{K}$ is an epimorphism, because it is a cover (Axiom \ref{axiom:subcan}).  We conclude that
  \begin{equation*}
    qd_i\pr_1g=qd_i\pr_2g,
  \end{equation*}
  or equivalently, that \ref{eq:difat} commutes.

  {\bf Step 3a: the induction.}

  We now construct \ref{eq:diseq}.  In detail, a section of $\mathbb{K}$ is a tuple $(x_0,x_1,(p_j)_{0\le j\ne i}^{n+1})$ where
  \begin{align*}
    x_k&\in X_{n+1},\text{ and}\\
    p_j&\in P^{\ge n}(f)_1,\intertext{such that}
    f(x_0)&=f(x_1),\intertext{and, for $k=0,1$, and all $j$,}
    d_{n+k}p_j&=d_j x_k.
  \end{align*}
  Equivalently, a section of $\mathbb{K}$ is a tuple $(x_0,x_1,(p_j)_{j=0}^{n+1})$ where $(x_0,x_1,(p_j)_{0\le j\ne i}^{n+1})$ satisfies the conditions above, and $p_i=s_nd_ix_1$. Note further that, for $i<n$
  \begin{align}
    d_i p_j &= s_{n-1}d_{n-1}d_i p_j\nonumber\\
        &= s_{n-1}d_n d_i p_j\nonumber\\
        &= s_{n-1}d_i d_{n+1}p_j\nonumber\\
        &= s_{n-1}d_i d_j x_1.\label{eq:forkhorn}
  \end{align}

  For the base of the induction, Equation \eqref{eq:forkhorn} implies that the assignment
  \begin{equation*}
    (x_0,x_1,(p_j)_{j=0}^{n+1})\mapsto((d_0s_nx_1,\ldots,d_{n-1}s_nx_1,-,x_1,p_{n+1}),s_nf(x_1))
  \end{equation*}
  defines a map $\mathbb{K}\to\Lambda^{n+2}_n(f)$. Denote by $K_1$ the pullback
  \begin{equation*}
    K_1:=\mathbb{K}\times_{\Lambda^{n+2}_n(f)}X_{n+2}
  \end{equation*}
  The projection $K_1\to\mathbb{K}$ is a cover, because it is the pullback of the cover $\lambda^{n+2}_n(f)$. Similarly, if we denote a section of $K_1$ by $(x_0,x_1,(p_j)_{j=0}^{n+1},z_{n+1})$, then Equation \eqref{eq:forkhorn} implies that the assignment
  \begin{equation*}
    (x_0,x_1,(p_j)_{j=0}^{n+1},z_{n+1})\mapsto((d_0s_{n+1}x_1,\ldots,d_{n-1}s_{n+1}x_1,p_n,-,d_n z_{n+1}),s_nf(x_1))
  \end{equation*}
  defines a map $K_1\to\Lambda^{n+2}_{n+1}(f)$. Denote by $K_2$ the pullback
  \begin{equation*}
    K_2:=K_1\times_{\Lambda^{n+2}_{n+1}(f)}X_{n+2}
  \end{equation*}
  The projection $K_2\to K_1$ is a cover, because it is the pullback of the cover $\lambda^{n+2}_{n+1}(f)$.

  For the inductive step suppose that we have constructed a cover $K_\ell\to K_{\ell-1}$, such that sections of $K_\ell$ are tuples $(x_0,x_1,(p_j)_{j=0}^{n+1},(z_j)_{j=n+2-\ell}^{n+1})$ with
    \begin{align*}
        (x_0,x_1,(p_j)_{j=0}^{n+1},(z_j)^{n+1}_{n+3-\ell})&\in K_{\ell-1},\text{ and}\\
        z_{n+2-\ell}&\in X_{n+2},\intertext{such that}
        f(z_{n+2-\ell})&=s_{n+1}d_{n+2}f(z_{n+2-\ell}),\text{ and}\\
        d_i z_{n+2-\ell}&=\left\{
            \begin{array}{rl}
                d_is_{n+1}d_{n+1}z_{n+3-\ell} & i\neq n+2-\ell,n+1\\
                p_{n+2-\ell} & i=n+2-\ell \\
                d_n z_{n+3-\ell} & i=n+1
            \end{array}
            \right.
    \end{align*}
    Then Equation \eqref{eq:forkhorn} implies that the assignment
    \begin{equation*}
        \begin{xy}
            \morphism/|->/<1000,0>[(x_0,x_1,(p_j)_{j=0}^{n+1},(z_j)_{j=n+2-\ell}^{n+1})`;]
            \morphism(500,-250)/{}/[((d_0s_{n+1}d_{n+1}z_{n+2-\ell},\ldots,d_{n-\ell}s_{n+1}d_{n+1}z_{n+2-\ell},p_{n+1-\ell},d_{n+2-\ell}s_{n+1}d_{n+1}z_{n+2-\ell},`;]
            \morphism(500,-500)/{}/[`\ldots,d_ns_{n+1}d_{n+1}z_{n+2-\ell},-,d_{n+1}z_{n+2-\ell}),s_{n+1}d_{n+1}f(z_{n+2-\ell}));]
        \end{xy}
    \end{equation*}
    defines a map $K_\ell\to\Lambda^{n+2}_{n+1}(f)$. Denote by $K_{\ell+1}$ the pullback
    \begin{equation*}
        K_{\ell+1}:=K_\ell\times_{\Lambda^{n+2}_{n+1}(f)}X_{n+2}
    \end{equation*}
    This completes the induction step for $\ell<n+2$.

    The induction above demonstrates the existence of the sequence of covers
    \begin{equation*}
        K_{n+2}\to\cdots \to K_0=\mathbb{K}
    \end{equation*}
    The construction allows us to denote a section of $K_{n+2}$ by
    \begin{align*}
        &(x_0,x_1,(p_j)_{0\le j\ne i}^{n+1},(z_j)_{j=0}^{n+1})\intertext{where}
        &s_{n+1}f(x_0)=f(z_0),\intertext{and, for $j\ne i$,}
        &d_jx_0=d_jd_{n+1}z_0.
    \end{align*}
    For $i\ne n+1$, the assignment
    \begin{multline*}
        (x_0,x_1,(p_j)_{0\le j\ne i}^{n+1},(z_j)_{j=0}^{n+1})\mapsto\\
            ((d_0s_{n+1}x_0,\ldots,d_{i-1}s_{n+1}x_0,-,d_{i+1}s_{n+1}x_0,\ldots,d_ns_{n+1}x_0,x_0,d_{n+1}z_0),s_{n+1}f(x_0))
    \end{multline*}
    defines a map $K_{n+2}\to\Lambda^{n+2}_i(f)$. Denote by $K_{n+3}$ the pullback
    \begin{equation*}
        K_{n+3}:=K_{n+2}\times_{\Lambda^{n+2}_i(f)}X_{n+2}
    \end{equation*}

    Similarly, for $i=n+1$, the assignment
    \begin{equation*}
        (x_0,x_1,(p_j)_{j=0}^n,(z_j)_{j=0}^n)\mapsto((d_0s_nx_0,\ldots,d_{n-1}s_nx_0,x_0,d_{n+1}z_0,-),s_nf(x_0))
    \end{equation*}
    defines a map $K_{n+2}\to\Lambda^{n+2}_{n+2}(f)$. Denote by $K_{n+3}$ the pullback
    \begin{equation*}
        K_{n+3}:=K_{n+2}\times_{\Lambda^{n+2}_{n+2}(f)}X_{n+2}
    \end{equation*}
    Note that, for any $i$, the map $K_{n+3}\to K_{n+2}$ is a cover, because it is a pullback of the cover $\lambda^{n+2}_j(f)$ (where $j=i$ if $i<n+1$ and $j=n+2$ if $i=n+1$).  Denote a section of $K_{n+3}$ by $(\mathbf{x},w)$ where $\mathbf{x}\in K_{n+2}$ and $w\in X_{n+2}$.  The construction, in all cases, guarantees that the assignment
    \begin{equation*}
        (\mathbf{x},w)\mapsto d_iw
    \end{equation*}
    defines a map $h\colon K_{n+3}\too P^{\ge n}(f)_1$ which, along with $c\colon K_{n+3}\too\mathbb{K}$, gives \ref{eq:diK}.
\end{proof}

By Axiom \ref{axiom:subcan}, \ref{eq:di} determines a map
\begin{equation*}
    \Lambda^{n+1}_i(\tn(f)) \to^{d_i} \tn(X,f)_n ,
\end{equation*}
We now extend \ref{eq:ntn} to a map
\begin{equation*}
    \tr_{n+1}\tn(f)\colon\tr_{n+1}\tn(X,f)_\bullet\to\tr_{n+1}Y_\bullet
\end{equation*}
On $k$-simplices, for $k\le n$, the map $\tr_{n+1}\tn(f)$ equals the map $\tr_n\tn(f)$.  On $(n+1)$-simplices, $\tr_{n+1}\tn(f)$ is the canonical map
\begin{equation*}
    \Lambda^{n+1}_1(\tn(f))\to Y_{n+1}
\end{equation*}
The missing face map
\begin{equation*}
    \tr_{n+1}\tn(X,f)_{n+1}=\Lambda^{n+1}_1(\tn(f))\to^{d_1}\pi_0(P^{\ge n}(f))=\tr_{n+1}\tn(X,f)_n
\end{equation*}
makes $\tr_{n+1}\tn(X,f)_\bullet$ into an $(n+1)$-truncated simplicial object.

\begin{definition}
  Let $f\colon X_\bullet\to Y_\bullet$ be an $\infty$-stack such that
  $\pi_0(P^{\ge n}(f))$ exists.  Define $\tn(X,f)_\bullet$ to be the limit
  \begin{equation*}
    \tn(X,f)_\bullet:=\csk_{n+1}\tr_{n+1}\tn(X,f)_\bullet\times_{\Csk_{n+1}Y_\bullet}Y_\bullet
  \end{equation*}
  The \emph{$n$-strictification of $f$} is the map
  \begin{equation*}
    \tn(f)\colon\tn(X,f)_\bullet\to Y_\bullet
  \end{equation*}
\end{definition}

\begin{proposition}\label{prop:tn}
  Let $f\colon X_\bullet\to Y_\bullet$ be an $\infty$-stack such that
  $\pi_0(P^{\ge n}(f))$ exists.  The $n$-strictification
  $\tn(f)\colon \tn(X,f)_\bullet\to Y_\bullet$ is an $n$-stack. The
  maps $f$ and $\tn(f)$ are isomorphic if and only if $f$ is an
  $n$-stack.
\end{proposition}
\begin{proof}
  We show that $\tn(f)$ is an $n$-stack. Lemma \ref{lemma:di1} established that $\lambda^k_i(\tn(f))$ is a cover for $k\le n$ and all $i$.  We now show that the map $\lambda^{n+1}_i(\tn(f))$ is an isomorphism for all $i$.  The inclusion
  $\Lambda^{n+1}_i\hookrightarrow\partial\Delta^{n+1}$ induces a map
  \begin{equation*}
    d_{\hat{\imath}}\colon M_{n+1}(\tn(f))\to\Lambda^{n+1}_i(\tn(f))
  \end{equation*}
  Observe that, in the notation of \ref{eq:Jintj}, if $J=[n+1]\setminus\{i\}$, then
  \begin{equation*}
    M_{n+1}(\tn(f))\cong \Lambda^{n+1}_i(\tn(f))\times_{\Lambda^{n+1}_{J\cap i}(f)}\pi_0(P^{\ge n}(f))
  \end{equation*}
  The missing face map $d_i\colon\Lambda^{n+1}_i(\tn(f))\to\pi_0(P^{\ge n}(f))$ induces a map
  \begin{equation*}
    (1,d_i)\colon\Lambda^{n+1}_i(\tn(f))\to M_{n+1}(\tn(f))
  \end{equation*}
  Note that the map $(1,d_i)$ is a right inverse for the map $d_{\hat{\imath}}$.

  For all $i$, the map $\lambda^{n+1}_i(\tn(f))$ factors as the composite
    \begin{equation*}
        \begin{xy}
            \morphism|a|<1250,0>[\tn(X,f)_{n+1}=\Lambda^{n+1}_1(\tn(f))`M_{n+1}(\tn(f));(1,d_1)]
            \morphism(1250,0)|a|<1000,0>[M_{n+1}(\tn(f))`\Lambda^{n+1}_i(\tn(f));d_{\hat{\imath}}]
        \end{xy}
    \end{equation*}
  These maps fit into a commuting diagram
  \begin{equation*}
    \begin{xy}
        \Atrianglepair|lmraa|/>`>`>`{@<3pt>}`{@<3pt>}/<1000,500>[X_{n+1}`\Lambda^{n+1}_1(\tn(f))`M_{n+1}(\tn(f))`\Lambda^{n+1}_i(\tn(f));```(1,d_1)`d_{\hat{\imath}}]
        \morphism(1000,0)|b|/{@<3pt>}/<-1000,0>[M_{n+1}(\tn(f))`\Lambda^{n+1}_1(\tn(f));d_{\hat{1}}]
        \morphism(2000,0)|b|/{@<3pt>}/<-1000,0>[\Lambda^{n+1}_i(\tn(f))`M_{n+1}(\tn(f));(1,d_i)]
    \end{xy}
  \end{equation*}
  The proof of Lemma \ref{lemma:di2} shows that
  \begin{align*}
    (1,d_1)\circ d_{\hat{1}}\circ(1,d_i)&=(1,d_i)\\
    (1,d_i)\circ d_{\hat{\imath}}\circ(1,d_1)&=(1,d_1)
  \end{align*}
  We conclude that
  \begin{align*}
    (d_{\hat{1}}(1,d_i))\circ\lambda^{n+1}_i(\tn(f))&=(d_{\hat{1}}(1,d_i))\circ(d_{\hat{\imath}}(1,d_1))\\
    &=d_{\hat{1}}\circ(1,d_1)\\
    &=1_{\Lambda^{n+1}_1(\tn(f))}.
  \end{align*}
  Similarly,
  \begin{align*}
    \lambda^{n+1}_i(\tn(f))\circ d_{\hat{1}}(1,d_i)&=d_{\hat{\imath}}(1,d_1)\circ d_{\hat{1}}(1,d_i)\\
    &=1_{\Lambda^{n+1}_i(\tn(f))}.
  \end{align*}
  We have shown that $\lambda^{n+1}_i(\tn(f))$ is an isomorphism for all $i$.

  Lemma \ref{lemma:psrep} guarantees that the limit $\Lambda^{n+2}_i(\tn(f))$ exists. Note that a section of $\Lambda^{n+2}_i(\tn(f))$ consists of a tuple
  \begin{equation*}
    \{x_j\}_{j\neq i}\in \tn(X,f)_{n+1}^{\times(n+2)}=\Lambda^{n+1}_1(\tn(f))^{\times(n+2)}\cong M_{n+1}(\tn(f))^{\times(n+2)}
  \end{equation*}
  such that, for $k<j$, $d_kx_j= d_{j-1}x_k$, while a section of $\tn(X,f)_{n+2}=M_{n+2}(\tn(f))$ consists of a tuple
  \begin{equation*}
    \{x_j\}\in \tn(X,f)_{n+1}^{\times(n+3)}=M_{n+1}(\tn(f))^{\times(n+3)}
  \end{equation*}
  such that, for $k<j$, $d_kx_j=d_{j-1}x_k$. In particular, the map
  \begin{equation*}
    \Lambda^{n+2}_i(\tn(f))\to^{\partial d_i} M_{n+1}(\tn(f))
  \end{equation*}
  induces an embedding
  \begin{align*}
    \Lambda^{n+2}_i(\tn(f))&\hookrightarrow M_{n+2}(\tn(f))=\tn(X,f)_{n+2}\\
    \{x_j\}_{j\neq i}&\mapsto \{x_j\}_{j\neq i}\cup\{x_i:=\partial d_i(\{x_j\}_{j\neq i})\}
  \end{align*}
  Because $\lambda^{n+1}_i(\tn(f))$ is an isomorphism for all $i$, this map is an isomorphism, with inverse given by $\lambda^{n+2}_i(\tn(f))$.

    For $k>n+2$, the map $\lambda^k_i(\tn(f))$ is an isomorphism because the map $\tn(f)$ is $(n+1)$-coskeletal, and the inclusion $\Lambda^k_i\hookrightarrow\Delta^k$ is the identity on $(n+1)$-skeleta.  By inductively applying Lemma \ref{lemma:psrep}, we conclude that $\tn(X,f)_k$ is an object of $\C$ for all $k$ and that $\tn(f)$ is an $n$-stack.

    We have shown that if $f$ is isomorphic to $\tn(f)$, then $f$ is an $n$-stack. Conversely, suppose that $f$ is an $n$-stack.  For any $k$ and any map, the $(k-1)$-skeleton of the map determines its $\Lambda^k_i$-horns.  The horn-filling maps for $\tn(f)$ are isomorphisms above dimension $n$.  If $f$ is an $n$-stack, then the horn-filling maps for $f$ are also isomorphisms above dimension $n$.  We conclude that the canonical map from $f$ to $\tn(f)$ is an isomorphism if $f$ is an $n$-stack and if the map from $f$ to $\tn(f)$ induces an isomorphism on $n$-skeleta.

    When $f$ is an $n$-stack, the map from $f$ to $\tn(f)$ is automatically an isomorphism on $n$-skeleta.  Indeed, if $f$ is an $n$-stack, then $P^{\geq n}(f)_\bullet$ is a Lie 0-groupoid (Theorem \ref{thm:kmor}).  As a result, the map from $X_n$ to $\pi_0(P^{\geq n}(f))$ is an isomorphism.
\end{proof}

If $f\colon X_\bullet\to Y_\bullet$ is a hypercover, then Proposition \ref{prop:hyppi0} and Theorem \ref{thm:kmorhyp} imply that
\begin{equation*}
    \pi_0(P^{\ge n}(f))\cong M_n(f)
\end{equation*}
In particular, the $n$-strictification $\tn(f)$ is an $n$-stack.

\begin{proposition}\label{prop:tnhyp}
  If $f:X_{\bullet}\rightarrow Y_{\bullet}$ is a hypercover, then
  \begin{enumerate}
  \item the map $\tn(f)$ is an $n$-hypercover, and
  \item the map
    \begin{xy}
      \morphism[X_{\bullet}`\tn(X,f)_{\bullet};q]
    \end{xy}
    is a hypercover.
  \end{enumerate}
\end{proposition}
\begin{proof}
  We show that $\tn(X,f)_{n+1}=\Lambda^{n+1}_1(\tn(f))\cong M_{n+1}(\tn(f))$.  Consider the commuting triangle
  \begin{equation*}
    \begin{xy}
        \Atriangle|lra|/>`>`{@<3pt>}/[X_{n+1}`\Lambda^{n+1}_1(\tn(f))`M_{n+1}(\tn(f));``(1,d_1)]
        \morphism(1000,0)|b|/{@<3pt>}/<-1000,0>[M_{n+1}(\tn(f))`\Lambda^{n+1}_1(\tn(f));d_{\hat{1}}]
    \end{xy}
  \end{equation*}
  We showed in Lemma \ref{lemma:di1} that the map $X_{n+1}\to\Lambda^{n+1}_1(\tn(f))$ is a cover.  A similar argument shows that, because $f$ is a hypercover, the map $X_{n+1}\to M_{n+1}(\tn(f))$ is a cover.  Axiom \ref{axiom:fg+g} implies that both $d_{\hat{1}}$ and $(1,d_1)$ are covers.  Axiom \ref{axiom:subcan} implies that they are both epimorphisms.  By construction
  \begin{equation*}
    d_{\hat{1}}(1,d_1)=1_{\Lambda^{n+1}_1(\tn(f))}
  \end{equation*}
  As a result,
  \begin{align*}
    (1,d_1)d_{\hat{1}}(1,d_1)&=(1,d_1)\\
    &=1_{M_{n+1}(\tn(f))}(1,d_1)
  \end{align*}
  Because $(1,d_1)$ is an epimorphism, we conclude that $(1,d_1)d_{\hat{1}}=1_{M_{n+1}(\tn(f))}$.

  This isomorphism combines with the isomorphism
  \begin{align*}
    \tn(X,f)_n&\cong M_n(f)\\
        &=M_n(\tn(f))
  \end{align*}
  to show that
  \begin{equation*}
    \tn(X,f)_\bullet\cong\Csk_{n-1}(X)_\bullet\times_{\Csk_{n-1}(Y)_\bullet}Y_\bullet
  \end{equation*}
  We have shown that $\tn(f)$ is an $n$-hypercover.

  For $k<n$, $M_k(q)\cong X_k$ by inspection. A similar check shows that the map from $M_n(q)$ to $\tn(X,f)_n$ is an isomorphism.  We observed above that $\tn(X,f)_n\cong M_n(f)$ when $f$ is a hypercover.  This implies that the map $X_n\to M_n(q)$ is isomorphic to the cover $X_n\to M_n(f)$.  For $k>n$, an exercise in combinatorics shows that the map $X_k\to M_k(q)$ is isomorphic to the cover $X_k\to M_k(f)$.  We conclude the proof.
\end{proof}

\begin{proposition}\label{prop:coc}
  Let $X_\bullet$ be a Lie $\infty$-groupoid, $U_\bullet\to^f X_\bullet$ a hypercover, and $Y_\bullet$ a Lie $n$-groupoid. Then any span
  \begin{equation*}
    \begin{xy}
      \Atriangle/>`>`{}/<600,0>[U_{\bullet}`X_\bullet`Y_{\bullet};f``]
    \end{xy}
  \end{equation*}
  factors uniquely through $\tn(f)$ as in the diagram
  \begin{equation*}
    \begin{xy}
      \Vtrianglepair/<-`>`<-`>`<-/<600,300>[X_\bullet`U_{\bullet}`Y_{\bullet}`\tn(U,f)_\bullet;f``\tn(f)``]
    \end{xy}
  \end{equation*}
  Further, the vertical map in this diagram is a hypercover.
\end{proposition}
\begin{proof}
  The data of a span is equivalent to a commuting triangle
  \begin{equation*}
    \begin{xy}
      \Vtriangle<600,300>[U_{\bullet}`X_\bullet\times Y_{\bullet}`X_{\bullet};`f`\pi_{X}]
    \end{xy}
  \end{equation*}
  Theorem \ref{thm:hk1} shows that both diagonal maps are
  $\infty$-stacks over $X_\bullet$.

  We apply $\tn$ to obtain the commuting diagram
  \begin{equation*}
    \begin{xy}
      \morphism(0,0)<1200,0>[\tn(U,f)_{\bullet}`\tn(X\times
      Y,\pi_{X})_{\bullet};]
      \morphism(0,0)/@{>}|!{(0,-300);(1200,-300)}\hole/<600,-600>[\tn(U,f)_{\bullet}`X_{\bullet};]
      \morphism(1200,0)/@{>}|!{(0,-300);(1200,-300)}\hole/<-600,-600>[\tn(X\times Y,\pi_{X})_{\bullet}`X_{\bullet};]
      \morphism(0,-300)<0,300>[U_{\bullet}`\tn(U,f)_{\bullet};]
      \morphism(1200,-300)<0,300>[Y_{\bullet}`\tn(X\times
      Y,\pi_{X})_{\bullet};]
      \morphism(0,-300)<1200,0>[U_{\bullet}`Y_{\bullet};]
      \Vtriangle(0,-600)<600,300>[U_{\bullet}`Y_{\bullet}`X_{\bullet};``]
    \end{xy}
  \end{equation*}
  Proposition \ref{prop:tn} shows that the map
  \begin{equation*}
    \begin{xy}
      \morphism<1250,0>[X_{\bullet}\times Y_{\bullet}`\tn(X\times Y,\pi_{X})_{\bullet};]
    \end{xy}
  \end{equation*}
  is an isomorphism.  Proposition \ref{prop:tnhyp} shows that $\tn(f)$
  is an $n$-hypercover and that the map
  \begin{equation*}
    \begin{xy}
      \morphism[U_{\bullet}`\tn(U,f)_{\bullet};]
    \end{xy}
  \end{equation*}
  is a hypercover.
\end{proof}

\section{n-Bundles and Descent}\label{sec:descent}
A classical construction produces a principal bundle for a Lie group
from local data on the base.  This local data is frequently presented
in the form of a cocycle $\varphi$ on a cover $U\rightarrow X$. If we
pass to the nerve of the cover $f:U_{\bullet}\rightarrow X$, we can
encode the cocycle as a 0-stack
\begin{equation}\label{eq:0bun}
  \begin{xy}
    \morphism[E^{\varphi}_{\bullet}`U_{\bullet};p]
  \end{xy}
\end{equation}
The principal bundle corresponding to the cocycle is the 0-strictification
\begin{equation*}
  \begin{xy}
    \morphism<750,0>[\tau_0(E^{\varphi},fp)`X;\tau_0(fp)]
  \end{xy}
\end{equation*}
From this perspective, the principal bundle is representable because
the 0-stack $p$ has the structure of a \emph{twisted Cartesian
  product}.  Twisted Cartesian products were studied by Barratt,
Guggenheim and Moore in their work on principal and associated bundles
for simplicial groups \cite{BGM:59}.
\begin{definition}\label{def:locn}
  A map $p:E_{\bullet}\rightarrow X_{\bullet}$ is a \emph{twisted
    Cartesian product}, if there exists $Y_\bullet\in\sC$, and there exist isomorphisms
    \begin{equation*}
      \begin{xy}
        \Vtriangle<600,300>[E_k`X_k\times Y_k`X_k;\varphi_k`p`\pi_{X_k}]
        \morphism(0,300)|b|<1200,0>[E_k`X_k\times Y_k;\cong]
      \end{xy}
    \end{equation*}
    for each $k\in\mathbb{N}$, such that, for $i<k$,
    \begin{equation*}
      \varphi_{k-1} d_i^E= (d_i^X\times d_i^Y)\varphi_k,
    \end{equation*}
    and, for all $i$,
    \begin{equation*}
      \varphi_k s_i^E=(s_i^X\times s_i^Y)\varphi_k.
    \end{equation*}
\end{definition}

\begin{definition}
  A \emph{local $n$-bundle} is a twisted Cartesian product which is also an $n$-stack.
\end{definition}

In analogy with \ref{eq:0bun}, we consider local $n$-bundles on the total spaces of
$(n+1)$-hypercovers
\begin{equation*}
  \begin{xy}
    \morphism[E_{\bullet}`U_{\bullet};p]
    \morphism(500,0)[U_{\bullet}`X_{\bullet};f]
  \end{xy}
\end{equation*}
Our goal in this section is to show that the $n$-strictification
\begin{equation*}
  \begin{xy}
    \morphism<750,0>[\tn(E,fp)_{\bullet}`X_{\bullet};\tn(fp)]
  \end{xy}
\end{equation*}
exists.

From the definition, it suffices to show that the orbit space $\pi_0 P^{\geq n}(fp)$ exists. Because the map $fp$ is an $(n+1)$-stack, the relative
higher morphism space $P^{\geq n}(fp)$ is a Lie 1-groupoid (Theorem \ref{thm:kmor}).  We see that the existence of the $n$-strictification is equivalent to the existence of the orbit space of a Lie groupoid.

\begin{theorem}[Godemont]
  Let $\G$ be a Lie groupoid in the category of analytic
  manifolds over a complete normed field.  If the map
  $(s,t):\G_1\rightarrow \G_0\times\G_0$ is a closed embedding, then
  $\pi_0(\G)$ is an analytic manifold and the map
  $\G_0\rightarrow\pi_0(\G)$ is a surjective submersion.
\end{theorem}
Serre gives a proof in \cite[Theorem III.12.2]{Ser:64} which applies mutatis mutandi to the category of smooth manifolds.  Inspired by this, we formulate an analogue of Godemont's Theorem for categories with covers.  We begin by defining an analogue of closed embeddings.
\begin{definition}
    Let $f\colon X\to Y$ be a morphism in $\C$. The \emph{graph} $\Gamma_f$ of $f$ is the inclusion
    \begin{equation*}
        \begin{xy}
            \morphism/^{ (}->/<750,0>[X\times_Y Y`X\times Y;\Gamma_f]
        \end{xy}
    \end{equation*}
    The subcategory of \emph{regular embeddings} is the smallest sub-category of $\C$ which is closed under pullback along covers and which contains all graphs.
\end{definition}

An isomorphism is a regular embedding, because it is a pullback of the graph
\begin{equation*}
    \begin{xy}
        \morphism[*`*\times *;\Delta]
    \end{xy}
\end{equation*}
along a cover $X\rightarrow *$.  A regular embedding in the category of smooth manifolds is a certain type of closed embedding.

\begin{definition}
  A \emph{regular} Lie $n$-groupoid is a Lie $n$-groupoid $X_\bullet$
  such that the map $\mu_1(X):X_1\to M_1(X)\cong X_0\times X_0$ is a
  regular embedding.
\end{definition}

\begin{axiom}\label{axiom:gode}(Godemont's Theorem)
  Let $X_\bullet$ be a regular Lie $n$-groupoid.  The orbit space $\pi_0(X_\bullet)$ exists.
\end{axiom}

\begin{remark}
    In the setting of smooth or analytic manifolds, this axiom is just a restatement of Godemont's theorem: the definition of a Lie $n$-groupoid ensures that the map $(d_0,d_1)\colon X_1\to X_0\times X_0$ defines an equivalence relation on $X_0$, and that $\pi_0(X_\bullet)$ is the quotient of $X_0$ by this relation. Serre's treatment in \cite{Ser:64} shows that $\pi_0(X_\bullet)$ is a manifold if $X_1\to X_0\times X_0$ is a closed embedding.
\end{remark}

\begin{theorem}[Descent for $n$-bundles]
  \label{thm:desc}
  Suppose that Godemont's Theorem holds in $\C$.  If
  $p:E_{\bullet}\rightarrow U_{\bullet}$ is a local $n$-bundle and
  $f:U_{\bullet}\rightarrow X_{\bullet}$ is an $(n+1)$-hypercover,
  then the $n$-strictification
  \begin{equation*}
    \begin{xy}
      \morphism<750,0>[\tn(E,fp)_{\bullet}`X_{\bullet};\tn(fp)]
    \end{xy}
  \end{equation*}
  exists.
\end{theorem}

\begin{remark}
    To study bundles for simplicial Lie groupoids, one should use twisted \emph{fiber} products rather than twisted Cartesian products in the definition of local $n$-bundles.  The theorem also holds for this more general notion.
\end{remark}

The following lemma is the crux of the proof.
\begin{lemma}
    There exists an isomorphism
    \begin{equation*}
        \begin{xy}
            \morphism(0,0)|a|<1000,0>[P^{\geq n}(fp)_1`P^{\geq n}(f)_1\times Y_n;\psi]
            \morphism(0,0)|b|<1000,0>[P^{\geq n}(fp)_1`P^{\geq n}(f)_1\times Y_n;\cong]
        \end{xy}
    \end{equation*}
    such that the maps $\phi_n d_n$ and $(d_n\times 1_{Y_n})\psi$ are equal.
\end{lemma}
\begin{proof}
  The definition of $P^{\geq n}(fp)_1$ allows us to view it as a
  sub-object of $E_{n+1}$.  Since $p$ is a twisted Cartesian product,
  there exists $Y_\bullet\in\sC$ and isomorphisms
    \begin{equation*}
      \begin{xy}
        \morphism(0,0)|a|<750,0>[E_k`U_k\times Y_k;\varphi_k]
        \morphism(0,0)|b|<750,0>[E_k`U_k\times Y_k;\cong]
      \end{xy}
    \end{equation*}
    such that, for $i<k$,
    \begin{equation*}
      \varphi_{k-1} d_i^E= (d_i^U\times d_i^Y)\varphi_k,
    \end{equation*}
    and, for all $i$,
    \begin{equation*}
      \varphi_k s_i^E=(s_i^U\times s_i^Y)\varphi_k.
    \end{equation*}
    Using this, we see that sections of $P^{\geq n}(fp)_1$ consist of pairs
    \begin{align*}
        (u,y)&\in U_{n+1}\times Y_{n+1}\intertext{such that, for $i<n$,}
        d_iy&=s_{n-1}d_{n-1}d_iy,\\
        d_iu&=s_{n-1}d_{n-1}d_iu,\intertext{and}
        f(u)&=s_nd_nf(u).
    \end{align*}
    Similarly, sections of $P^{\geq n}(p)_1$ consist of pairs
    \begin{align*}
        (u,y)&\in U_{n+1}\times Y_{n+1}\intertext{such that, for $i<n$,}
        d_iy&=s_{n-1}d_{n-1}d_iy,\intertext{and}
        u&=s_nd_nu.
    \end{align*}
    These equations say that if $(u,y)$ is a section of $P^{\geq n}(fp)_1$, then $u$ is a section of $P^{\geq n}(f)_1$ and the natural map
    \begin{equation}\label{eq:descpfmp1}
        \begin{xy}
            \morphism(0,0)<1000,0>[P^{\geq n}(fp)_1`P^{\geq n}(f)_1\times_{U_n}P^{\geq n}(p)_1;]
            \morphism(0,-200)/|->/<1000,0>[(u,y)`(u,(s_nd_nu,y));]
        \end{xy}
    \end{equation}
    is an isomorphism.

    Because $p$ is an $n$-stack, $P^{\geq n}(p)_{\bullet}$ is a Lie $0$-groupoid (Theorem \ref{thm:kmor}).  The target map
    \begin{equation*}
        \begin{xy}
            \morphism<750,0>[P^{\geq n}(p)_1`P^{\geq n}(p)_0;]
        \end{xy}
    \end{equation*}
    is therefore an isomorphism.  Using this isomorphism and the map \ref{eq:descpfmp1}, we obtain the desired isomorphism
    \begin{equation*}
        \begin{xy}
            \morphism(0,0)|a|<1000,0>[P^{\geq n}(fp)_1`P^{\geq n}(f)_1\times Y_n;\psi]
            \morphism(0,0)|b|<1000,0>[P^{\geq n}(fp)_1`P^{\geq n}(f)_1\times Y_n;\cong]
        \end{xy}
    \end{equation*}
\end{proof}

\begin{proof}[Proof of Theorem \ref{thm:desc}]
  Axiom \ref{axiom:gode} reduces the proof to showing that
  \begin{equation*}
    \begin{xy}
        \morphism<1250,0>[P^{\geq n}(fp)_1`P^{\geq n}(fp)_0\times P^{\geq n}(fp)_0;]
    \end{xy}
  \end{equation*}
  is a regular embedding. The map
    \begin{equation*}
        \begin{xy}
            \morphism<750,0>[P^{\geq n}(f)_{\bullet}`M_n(f);\pi]
        \end{xy}
    \end{equation*}
    is a 1-hypercover, because $f$ is an $(n+1)$-hypercover (Theorem \ref{thm:kmorhyp}).  In particular, the map $\mu_1(\pi)$ gives an isomorphism
    \begin{equation*}
        \begin{xy}
            \morphism(0,0)|a|<1000,0>[P^{\geq n}(f)_1`U_n\times_{M_n(f)}U_n;\mu_1(\pi)]
            \morphism(0,0)|b|<1000,0>[P^{\geq n}(f)_1`U_n\times_{M_n(f)}U_n;\cong]
        \end{xy}
    \end{equation*}
    The canonical map
    \begin{equation*}
        \begin{xy}
            \morphism(0,0)/^{ (}->/<1000,0>[U_n\times_{M_n(f)}U_n
            `U_n\times U_n;\imath]
        \end{xy}
    \end{equation*}
    is a regular embedding.  Composing this with the map above, we obtain a regular embedding
    \begin{equation*}
        \begin{xy}
            \morphism<1000,0>[P^{\geq n}(f)_1`U_n\times U_n;\imath\mu_1(\pi)]
        \end{xy}
    \end{equation*}
    The isomorphisms
    \begin{equation*}
        \begin{xy}
            \morphism<750,0>[P^{\geq n}(fp)_0`E_n;\cong]
            \morphism(750,0)<750,0>[E_n`U_n\times Y_n;\cong]
        \end{xy}
    \end{equation*}
    and
    \begin{equation*}
        \begin{xy}
            \morphism(0,0)|a|<1000,0>[P^{\geq n}(fp)_1`P^{\geq n}(f)_1\times Y_n;\psi]
            \morphism(0,0)|b|<1000,0>[P^{\geq n}(fp)_1`P^{\geq n}(f)_1\times Y_n;\cong]
        \end{xy}
    \end{equation*}
    allow us to factor the map
    \begin{equation*}
        \begin{xy}
            \morphism<1200,0>[P^{\geq n}(fp)_1`P^{\geq n}(fp)_0\times P^{\geq n}(fp)_0;]
        \end{xy}
      \end{equation*}
      as the composite
      \begin{multline*}
        \begin{xy}
          \morphism<900,0>[P^{\geq n}(fp)_1`P^{\geq n}(fp)_1\times Y_n;%
          \Gamma_{d_{n+1}^E}]
          \morphism(900,0)<1200,0>[P^{\geq n}(fp)_1\times Y_n`%
          Y_n\times P^{\geq n}(f)_1\times Y_n;%
          \psi\times 1_Y]
        \end{xy} \\
        \begin{xy}
          \morphism(2100,0)<1400,0>[`Y_n\times U_n\times U_n\times Y_n;%
          1_Y\times(\imath\mu_1(\pi))\times 1_Y]
        \end{xy}
      \end{multline*}
      Each map in this sequence is a regular embedding.
\end{proof}

\section{Strict Lie n-Groups and their Actions}\label{sec:slie}
Principal and associated bundles for discrete simplicial groups provide examples of local $n$-bundles in $\sst$.  This theory was developed by Barratt, Guggenheim and Moore \cite{BGM:59}.  In this section, we develop analogous results for simplicial Lie groups.  While we restrict to simplicial groups for ease of exposition, the results and proofs carry over to simplicial Lie groupoids. We refer the reader to \cite[Section 4]{NSS:15} for a related treatment of principal bundles for simplicial Lie groups.

\begin{definition}
    A \emph{simplicial Lie group} $G_{\bullet}$ in $\C$ is a simplicial diagram in the category of group objects in $\C$.  Denote by $\sgp(\C)$ the category of simplicial Lie groups.
\end{definition}

Eilenberg and Mac Lane \cite{EiM:53} introduced a pair of functors $\W$ and $\LW$ from simplicial groups to simplicial sets which generalize the universal bundle and nerve of a group.

\begin{definition}
  Let $G_{\bullet}$ be a simplicial group.
  \begin{enumerate}
  \item The \emph{total space $\W_{\bullet}G$ of the universal
      $G_{\bullet}$-bundle} is the simplicial set with
    \begin{align*}
      \W_nG&:= G_0\times\cdots\times G_n\\
      d_i(g_0,\ldots, g_n)&:=(g_0,\ldots,g_{i-2},g_{i-1}d_ig_i,d_ig_{i+1},\ldots,d_ig_n)\\
      s_i(g_0,\ldots, g_n)&:=(g_0,\ldots,g_{i-1},e,s_ig_i,\ldots,s_ig_n)
    \end{align*}
  \item The \emph{nerve} $\LW_{\bullet}G$ of $G_{\bullet}$ is the
    simplicial set with
    \begin{align*}
        \LW_0G&:=\ast,\intertext{and, for $n>0$,}
        \LW_nG &:= G_0\times\cdots\times G_{n-1} \\
        d_i(g_0,\ldots, g_{n-1}) &:=
        (g_0,\ldots,g_{i-2},g_{i-1}d_ig_i,d_ig_{i+1},\ldots,d_ig_{n-1}) \\
        s_i(g_0,\ldots, g_{n-1}) &:=
        (g_0,\ldots,g_{i-1},e,s_ig_i,\ldots,s_ig_{n-1})
    \end{align*}
  \item The assignment which sends
    $(g_0,\ldots,g_n)\in\W_nG$ to $(g_0,\ldots,g_{n-1})\in\LW_nG$
    defines a twisted Cartesian product
    $\W_{\bullet}G\rightarrow\LW_{\bullet}G$.  This is the
    \emph{universal $G_{\bullet}$-bundle}.
  \end{enumerate}
\end{definition}
Because $\C$ has finite products, the same formulas give functors $\W$
and $\LW$ from the category of simplicial Lie groups to the
category $\sC$.

\begin{definition}
  If $G_{\bullet}$ is a simplicial Lie group and $X_\bullet$ is a simplicial object in $\C$, then a left
  \emph{left action $G_\bullet\circlearrowleft X_\bullet$} consists of maps
  \begin{equation*}
    \begin{xy}
      \morphism(0,0)<750,0>[G_k\times X_k`X_k;]
      \morphism(0,-200)/|->/<750,0>[(g,x)`gx;]
    \end{xy}
  \end{equation*}
  for each $k$, such that, in all dimensions and for all $i$:
    \begin{align*}
        g_1(g_2x)&=(g_1g_2)x,\\
        ex&=x,\\
        d_i(gx)&=(d_ig)(d_ix),\text{ and}\\
        s_i(gx)&=(s_ig)(s_ix).
    \end{align*}
    Right actions are defined analogously.
\end{definition}
When $G_{\bullet}$ and $X_{\bullet}$ are constant simplicial diagrams, this is the usual notion of a left action of a Lie group. The right action of $G_{\bullet}$ on itself induces a right $G_{\bullet}$-action on $\W_\bullet G$.
\begin{definition}
    Suppose we have a left action $G_{\bullet}\circlearrowleft X_{\bullet}$.  The \emph{homotopy quotient} $(\W G\times_G X)_{\bullet}$ is defined by
    \begin{align*}
        (\W G\times_G X)_n &:= \LW_n G\times X_n\\
        s_i(g_0,\ldots, g_{n-1},x)&:=(g_0,\ldots,g_{i-1},e,s_ig_i,\ldots,s_ig_{n-1},s_ix)\\
        d_n(g_0,\ldots, g_{n-1},x)&:=(g_0,\ldots,g_{n-2},g_{n-1}d_nx)\intertext{and, for $i<n$}
        d_i(g_0,\ldots, g_{n-1},x)&:=(g_0,\ldots,g_{i-1}d_ig_i,\ldots,d_ig_{n-1},d_ix).
    \end{align*}
\end{definition}
If $G_{\bullet}$ acts on $X_{\bullet}$ and $Y_{\bullet}$ and $f:X_{\bullet}\rightarrow Y_{\bullet}$ is $G_{\bullet}$-equivariant, then $f$ induces a map of homotopy quotients
\begin{equation*}
    \begin{xy}
        \morphism<1000,0>[(\W G\times_G X)_{\bullet}`(\W G\times_G Y)_{\bullet};1\times_G f]
    \end{xy}
\end{equation*}
given on $n$-simplices by $1_{\LW G}\times f_n$.

\begin{definition}\label{def:slien}
    A \emph{strict Lie $n$-group} is a simplicial Lie group $G_{\bullet}$ such that the horn-filling maps $\lambda^k_i(G)$ are isomorphisms for $k\geq n$.
\end{definition}
We might have defined a strict Lie $n$-group as a simplicial Lie group such that the maps $\lambda^k_i(G)$ were also covers for all $k<n$.  An argument due to Moore shows that this follows from our definition.
\begin{proposition}
    Let $G_{\bullet}$ be a strict Lie $n$-group.  The simplicial object underlying $G_{\bullet}$ is a Lie $(n-1)$-groupoid.
\end{proposition}
\begin{proof}
    We perform an induction on the dimension of the horns to show that the maps $\lambda^k_i(G)$ are covers for $k<n$.

    Suppose that for $l<k$ and all $i$, the limit $\Lambda^l_i(G)$ exists and the map $\lambda^l_i(G)$ is a cover. Lemma \ref{lemma:psrep} shows that the limit $\Lambda^k_i(G)$ exists for all $i$.

    Sections of $\Lambda^k_i(G)\times G_k$ are tuples
    \begin{equation*}
        ((g_0,\ldots,-,\ldots,g_k),g)\in\Lambda^k_i(G)\times G_k
    \end{equation*}
    such that, for $j<m\neq i$,
    \begin{align*}
        d_{m-1}g_j&=d_jg_m.
    \end{align*}
    We perform an induction on $0\le\ell\le k+1$ to construct sections $g^\ell\in G_k$ such that for $j<\ell$ we have
    \begin{equation*}
        d_jg^\ell=g_j.
    \end{equation*}
    Fix $((g_0,\ldots,-,\ldots,g_k),g)\in\Lambda^k_i(G)\times G_k$, and set
    \begin{equation*}
        g^0:=g.
    \end{equation*}
    Now suppose that for $0\le\ell$, we have $g^\ell\in G_k$ such that, for $j<\ell$,
    \begin{equation*}
        d_jg^\ell=g_j.
    \end{equation*}
    We define
    \begin{equation*}
        a^\ell_j:=g_j(d_jg^\ell)^{-1}\in G_{k-1}.
    \end{equation*}
    The horn relations on the $g_j$ ensure that $(a^\ell_0,\ldots,-,\ldots,a^\ell_k)$ defines a section of $\Lambda^k_i(G)$, so we have
    \begin{equation*}
        ((a^\ell_0,\ldots,-,\ldots,a^\ell_k),g^l)\in\Lambda^k_i(G)\times G.
    \end{equation*}
    We define
    \begin{equation*}
        g^{l+1}:=(s_\ell a^\ell_\ell)g^\ell.
    \end{equation*}
    A short exercise shows that, for $j<\ell+1$,
    \begin{align*}
        d_jg^{\ell+1}=g_j.
    \end{align*}
    This completes the induction step.  We now define
    \begin{equation*}
        \begin{xy}
            \morphism(0,0)<1800,0>[\Lambda^k_i(G)\times G_k`\Lambda^k_i(G)\times G_k;\varphi]
            \morphism(0,-250)/|->/<1800,0>[((g_0,\ldots,-,\ldots,g_k),g)`((g_0(d_0g^{k+1})^{-1},\ldots,-,\ldots,g_k(d_kg^{k+1})^{-1}),g^{k+1});]
        \end{xy}
    \end{equation*}
    By construction, $\varphi$ is an isomorphism.  It factors the projection
    \begin{equation*}
        \begin{xy}
            \morphism<750,0>[\Lambda^k_i(G)\times G_k`\Lambda^k_i(G);]
        \end{xy}
    \end{equation*}
    as
    \begin{equation*}
        \begin{xy}
            \square/>`>`>`<-/<1150,500>[\Lambda^k_i(G)\times G_k`\Lambda^k_i(G)\times G_k`\Lambda^k_i(G)`G_k;\varphi``\pi_{G_k}`\lambda^k_i(G)]
            \morphism(0,500)|b|<1150,0>[\Lambda^k_i(G)\times G_k`\Lambda^k_i(G)\times G_k;\cong]
        \end{xy}
    \end{equation*}
    Axiom \ref{axiom:fg+g} guarantees that $\lambda^k_i(G)$ is a
    cover. Lemma \ref{lemma:psrep} shows that the limit $\Lambda^{k+1}_i(G)$ exists for all $i$.  This concludes the induction step.
\end{proof}

Observe that a strict Lie 1-group is a Lie group viewed as a constant simplicial diagram.  A strict Lie 2-group is a simplicial Lie group $G_{\bullet}$ such that the simplicial object underlying $G_{\bullet}$ is the nerve of a Lie groupoid.  Therefore, a strict Lie 2-group could be equivalently described as a Lie group in the category of Lie groupoids.  Strict Lie 2-groups are relatively abundant.  For example, the Lie 2-group associated to a finite-dimensional nilpotent differential graded Lie algebra concentrated in degrees $[-1,0]$ is a strict Lie 2-group (see \cite{Get:02} for details).

\begin{theorem}\label{thm:wbar}\mbox{}
  \begin{enumerate}
      \item The nerve of a strict Lie $n$-group is a Lie $n$-group.
      \item The homotopy quotient of a $G_{\bullet}$-equivariant $n$-stack is an $n$-stack.
  \end{enumerate}
\end{theorem}
\begin{proof}
  Let $G_{\bullet}$ be a strict $n$-group, or let
  $\varphi:X_{\bullet}\rightarrow Y_{\bullet}$ be a
  $G_{\bullet}$-equivariant $n$-stack.  We show that, in the first case,
  $\LW_{\bullet}G$ is a Lie $n$-group, and, in the second, that
  $1\times_{G}\varphi$ is an $n$-stack.

  Denote sections of $\LW_kG$ by
  \begin{align*}
    \g^j&=(g^j_0,\ldots,g^j_{k-1})\in
    G_0\times\cdots\times G_{k-1} =
    \LW_kG.
  \end{align*}
  Denote sections of $\Lambda^k_i(\LW G)$ by
  \begin{equation*}
    (\g^0,\ldots,-,\ldots,\g^k)\in \Lambda^k_i(\LW G).
  \end{equation*}
  We now give a series of isomorphisms which relate horns in the
  nerve or homotopy quotient to horns in the strict Lie
  $n$-group or equivariant $n$-stack.  The formulas proceed from
  the observation that the highest face in these horns determines all
  but the last coordinates of the lower ones; these last coordinates
  themselves determine a horn in the original simplicial Lie group
  or $G_{\bullet}$-map.

  For $k>0$, the maps $\Lambda^k_i(\LW G)\to
  \LW_{k-1}G\times\Lambda^{k-1}_i(G)$ given by
  \begin{equation*}
    (\g^0,\ldots,\widehat{\g^i},\ldots,\g^k) \mapsto
    \begin{cases}
      (\g^k,(g^0_{k-2},\ldots,-,\ldots,g^{k-2}_{k-2},(g^k_{k-2})^{-1}g^{k-1}_{k-2}))
      & i<k-1 , \\[5pt]
      (\g^k,(g^0_{k-2},\ldots,g^{k-2}_{k-2},-)) & i=k-1 , \\[5pt]
      (\g^{k-1},(g^0_{k-2},\ldots,g^{k-2}_{k-2},-)) & i=k ,
    \end{cases}
  \end{equation*}
  are isomorphisms. For $k>0$ and $i<k$, the maps
  $\Lambda^k_i(\W G\times_G\varphi)\to
  \LW_k G\times\Lambda^k_i(\varphi)$ given by
  \begin{multline*}
    (((\g^0,x_0),\ldots,\widehat{(\g^i,x_i)},\ldots,(\g^k,x_k)),(\h,y))
    \\
    \mapsto
    \begin{cases}
      (\h,((x_0,\ldots,\widehat{x_i},\ldots,x_{k-1},(h_{k-1})^{-1}x_k),y))
      & i<k , \\
      (\h,((x_0,\ldots,x_{k-1},-),y)) & i=k ,
    \end{cases}
  \end{multline*}
  are isomorphisms. For $k>1$, the isomorphisms for horns in the
  nerve fit into the following commuting squares:
  \begin{align*}
    & \begin{xy}
      \Square(0,0)/>`=`>`>/[\LW_k G`\Lambda^k_i(\LW G)`%
      \LW_{k-1}G\times G_{k-1}`%
      \LW_{k-1}G\times\Lambda^{k-1}_i(G);%
      \lambda^k_i(\LW G)``\cong`1\times\lambda^{k-1}_i(G)]
    \end{xy} \\
    & \begin{xy}
      \Square(0,0)/>`=`>`>/[\LW_k G`\Lambda^k_k(\LW G)`%
      \LW_{k-1}G\times G_{k-1}`%
      \LW_{k-1}G\times\Lambda^{k-1}_{k-1}(G);%
      \lambda^k_k(\LW G)``\cong`1\times\lambda^{k-1}_{k-1}(G)]
    \end{xy}
  \end{align*}
  For $k=1$, we have
  \begin{equation*}
    \begin{xy}
      \Square/>`=`>`>/[\LW_1 G`\Lambda^1_i(\LW G)`G_0`\ast;\lambda^1_i(\LW G)``\cong`]
    \end{xy}
  \end{equation*}
  Similarly, for $k\geq 1$, the isomorphisms for horns in the homotopy
  quotients fit into the commuting squares
  \begin{equation*}
    \begin{xy}
      \Square/>`=`>`>/[\W G\times_G X_k`%
      \Lambda^k_i(\W G\times_G\varphi)`%
      \LW_k G\times X_k`%
      \LW_k G\times\Lambda^k_i(\phi);%
      \lambda^k_i(1\times_G\varphi)``\cong`1\times\lambda^k_i(\varphi)]
    \end{xy}
  \end{equation*}
  These squares show that the relevant horn-filling maps for
  $\LW_{\bullet}G$ and $1\times_G\varphi$ are covers for
  all $k$ and isomorphisms for $k>n$.
\end{proof}

\begin{remark}
    While we do not need it for this paper, one could define a strict $n$-stack as a homomorphism of simplicial Lie groups such that the relative horn-filling maps are isomorphisms in dimensions at least $n$.  The analogues of the results above hold in the relative case, with minimal changes to the proofs.  One could also make analogous definitions for simplicial Lie groupoids.  The analogues of the results above hold, with minimal changes to the proofs.
\end{remark}

\section{A Finite Dimensional String 2-Group}
\label{sec:string}

In this section, we specialize to the category of smooth manifolds and
apply our results to construct finite dimensional Lie 2-groups.  Let
$A$ be an abelian group. For each natural number $n$, Eilenberg and
MacLane introduced a simplicial abelian group $K(A,n)_{\bullet}$ whose
geometric realization represents the cohomology functor $H^n(-;A)$.
They further observed that $\LW_{\bullet}K(A,n)$ is isomorphic to
$K(A,n+1)_{\bullet}$.  This construction and identification
also exist for abelian Lie groups.

\begin{definition}
  Let $G$ be a Lie group. Let $A$ an abelian Lie group.
  \begin{enumerate}
  \item An \emph{$A$-valued $n$-cocycle} on $\LW_{\bullet}G$ is a span
    \begin{equation*}
      \begin{xy}
        \Atriangle/>`>`{}/<600,0>[U_{\bullet}`\LW_{\bullet}G`K(A,n)_{\bullet};``]
      \end{xy}
    \end{equation*}
    such that $U_{\bullet}\rightarrow\LW_{\bullet}G$ is a hypercover.
  \item An \emph{equivalence of cocycles} is a commuting diagram of cocycles
    \begin{equation*}
      \begin{xy}
        \Atrianglepair(0,0)/>`<-`>`<-`>/<600,300>[U^0_{\bullet}`\LW_{\bullet}G`V_{\bullet}`K(A,n)_{\bullet};````]
        \Vtrianglepair(0,-300)/<-`>`<-`>`<-/<600,300>[\LW_{\bullet}G`V_{\bullet}`K(A,n)_{\bullet}`U^1_{\bullet};````]
      \end{xy}
    \end{equation*}
    such that the maps $V_{\bullet}\rightarrow U^i_{\bullet}$ are hypercovers.
  \end{enumerate}
\end{definition}

The connected $2$-types $\mathcal{G}_{\bullet}\in\ssm$ which have
arisen in the literature are determined by
\begin{enumerate}
\item a Lie group $G$,
\item an abelian Lie group $A$, and
\item an equivalence class of $3$-cocycles
  \begin{equation*}
    \begin{xy}
      \Atriangle/>`>`{}/<600,0>[U_{\bullet}`\LW_{\bullet}G`K(A,3)_{\bullet};``]
    \end{xy}
  \end{equation*}
\end{enumerate}
Much work has gone into finding geometric models for smooth $2$-types.
By pulling back the universal twisted $K(A,2)$-bundle along a cocycle,
we obtain a local $2$-bundle
\begin{equation*}
  \begin{xy}
    \morphism[E_{\bullet}`U_{\bullet};].
  \end{xy}
\end{equation*}
The composite
\begin{equation*}
  \begin{xy}
    \morphism[E_{\bullet}`U_{\bullet};]
    \morphism(500,0)[U_{\bullet}`\LW_{\bullet}G;]
  \end{xy}
\end{equation*}
is an $\infty$-stack.  This shows that connected smooth $2$-types can be
realized as finite dimensional Lie $\infty$-groups.

Over the last decade there have been many attempts to do better.  The most relevant of these is provided by Schommer-Pries
\cite{Sch:11} who showed that connected smooth $2$-types can be
realized as weak group objects in the bicategory of finite dimensional
Lie groupoids.  Zhu \cite{Zhu:09}, drawing on ideas of Duskin,
constructed a nerve for such weak group objects and showed that the
nerve is a Lie $2$-group.

The tools in this article allow us to construct a Lie $2$-group
$X_{\bullet}$ directly from the data above.  The object
produced is equivalent to the one obtained by Zhu from Schommer-Pries.  Our methods
extend to $n>2$.

By Theorem \ref{thm:wbar}, Proposition \ref{prop:coc} specializes to the following.
\begin{corollary}
  \label{cor:coc}
  Any $A$-valued $n$-cocycle
  \begin{equation*}
    \begin{xy}
      \Atriangle/>`>`{}/<600,0>[U_{\bullet}`\LW_{\bullet}G`K(A,n)_{\bullet};f``]
    \end{xy}
  \end{equation*}
  factors uniquely through $\tn(f)$ as in the diagram
  \begin{equation*}
    \begin{xy}
      \Vtrianglepair/<-`>`<-`>`<-/<600,300>[\LW_{\bullet}G`U_{\bullet}`K(A,n)_{\bullet}`\tn(U,f)_\bullet;f``\tn(f)``]
    \end{xy}
  \end{equation*}
  This factorization is an equivalence of cocycles.
\end{corollary}

We can now use Theorem \ref{thm:desc} to produce a Lie $2$-group from
an $A$-valued $3$-cocycle on $\LW_{\bullet}G$. We can assume, without
loss of generality, that the $3$-cocycle
\begin{equation}\label{eq:coc}
  \begin{xy}
    \Atriangle|aab|/>`>`{}/<600,0>[U_{\bullet}`\LW_{\bullet}G`K(A,3)_{\bullet};f`\varphi`]
  \end{xy}
\end{equation}
has $U_0=\ast$ and $f$ a $3$-hypercover.  We pull back the
universal $K(A,2)$-bundle along $\varphi$ to obtain a local $2$-bundle
\begin{equation*}
  \begin{xy}
    \morphism<750,0>[\varphi^{\ast}\W K(A,2)`U_{\bullet};p]
  \end{xy}
\end{equation*}
We descend this local $2$-bundle along the $3$-hypercover $f$, as in
Theorem \ref{thm:desc}. We obtain a $2$-stack
\begin{equation*}
  \begin{xy}
    \morphism<1240,0>[X_\bullet:=\tau_2(\varphi^{\ast}\W K(A,2),fp)_{\bullet}`\LW_{\bullet}G;]
  \end{xy}
\end{equation*}
The object $X_\bullet$ is the desired Lie $2$-group.

We now examine $X_{\bullet}$ in more detail.  For simplicity, we ignore
degeneracies.  The hypercover in the cocycle \ref{eq:coc} is
determined by its $2$-skeleton (Proposition \ref{prop:tnhyp}).  The
$2$-skeleton consists of
\begin{enumerate}
\item a cover $f_1:U\rightarrow G$, which we view as a 1-truncated hypercover, and
\item a cover $f_2:V\rightarrow(\csk_1 U\times_{\Csk_1\LW G}\LW G)_2$.
\end{enumerate}
The Lie $2$-group $X_{\bullet}$ has the same $1$-skeleton as this hypercover.

The $2$-simplices of $X_{\bullet}$ are determined by the Lie
$1$-groupoid $P^{\geq 2}(fp)_{\bullet}$.  Its vertex manifold,
$P^{\geq 2}(fp)_0$, is isomorphic to $V\times A$.  We showed that
\begin{align*}
  P^{\geq 2}(fp)_1&\cong (V\times_{(\csk_1 U\times_{\Csk_1\LW G}\LW G)_2}V)\times K(A,2)_2\\
  &\cong (V\times_{(\csk_1 U\times_{\Csk_1\LW G}\LW G)_2}V)\times A
\end{align*}
at the end of the proof of Theorem \ref{thm:desc}.  The target map of
$P^{\geq 2}(fp)_{\bullet}$ is given by
\begin{equation*}
  (v_2,v_3,a)\mapsto(v_2,a)
\end{equation*}
We abuse notation and denote by $\varphi$ the restriction of the map
\begin{equation*}
  \begin{xy}
    \morphism[U_3`K(A,3)_3;\varphi]
  \end{xy}
\end{equation*}
to $P^{\geq 2}(f)_1\subset U_3$.  The source in the Lie groupoid
$P^{\geq 2}(f)_{\bullet}$ is given by
\begin{equation*}
  (v_2,v_3,a) \mapsto (v_3,\varphi(v_2,v_3)+a)
\end{equation*}
The Lie $1$-groupoid structure on $P^{\geq 2}(fp)_{\bullet}$ ensures
that $\varphi$ is an $A$-valued 1-cocycle on the cover
\begin{equation*}
  \begin{xy}
    \morphism<1000,0>[V`(\csk_1 U\times_{\Csk_1\LW G}\LW G)_2;]
  \end{xy}
\end{equation*}
The orbit space
\begin{equation*}
  \begin{xy}
    \morphism<1000,0>[X_2`(\csk_1 U\times_{\Csk_1\LW G}\LW G)_2;]
  \end{xy}
\end{equation*}
is the principal $A$-bundle determined by this data.

We now describe the higher simplices.  The map
$\Lambda^k_i(\tau_2(fp))$ is an isomorphism for $k>2$ and all
$i$ (Proposition \ref{prop:tn}). For $k=3$, the data of these
isomorphisms can be reduced to a trivialization $\zeta$ of the bundle
\begin{equation*}
  \begin{xy}
    \morphism(0,0)<1500,0>[d_3^*P^{\vee}\otimes d_2^*P\otimes
    d_1^*P^{\vee}\otimes d_0^*P`(\csk_1 U\times_{\Csk_1\LW G}\LW G)_3;]
    \end{xy}
\end{equation*}
Here $d_i$ denotes the face map from $3$-simplices to $2$-simplices in the simplicial object
$(\csk_1 U\times_{\Csk_1\LW G}\LW G)_{\bullet}$, and $P^{\vee}$
denotes the dual bundle of $P$.

Proposition \ref{prop:tn} implies that $X_{\bullet}$ is determined by
its $3$-skeleton.  In the present context, this is equivalent to
requiring that $\zeta$ satisfy a pentagonal coherence condition coming
from the $1$-skeleton of the $4$-simplex.

Summing up, we see that the Lie $2$-group $X_{\bullet}$ is reducible
to the following data:
\begin{enumerate}
\item a cover $f\colon U\too G$,
\item a principal $A$-bundle $P\too(\csk_1 U\times_{\Csk_1\LW G}\LW G)_2$, and
\item a trivialization $\zeta$ of
  \begin{equation*}
    \begin{xy}
      \morphism(0,0)<1500,0>[d_3^*P^{\vee}\otimes d_2^*P\otimes d_1^*P^{\vee}\otimes d_0^*P`(\csk_1 U\times_{\Csk_1\LW G}\LW G)_3;]
    \end{xy}
  \end{equation*}
  which satisfies a pentagonal coherence condition coming from the
  1-skeleton of the 4-simplex.
\end{enumerate}
If we take $\spn$ for $G$, $U(1)$ for $A$, and a cocycle representing
the fractional first Pontrjagin class $\frac{p_1}{2}$, then the
resulting Lie $2$-group is the nerve, \textit{\`{a} la} Duskin--Zhu, of
Schommer-Pries's model for $\strn$.

The discussion above gives a construction of higher central extensions
of Lie groups.  More generally, we might consider higher
\emph{abelian} extensions.  We consider an abelian Lie group $A$ on
which the Lie group $G$ acts by automorphisms.  The data which
specifies a higher abelian extension of a Lie group, and the
construction which produces a Lie $2$-group from this data, are
analogous to the case of higher central extensions above.  We sketch
the necessary changes.

For each $n>0$, the $G$-action on $A$ induces $G$-actions on
$\W_{\bullet}K(A,n-1)$ and $K(A,n)_{\bullet}$ such that the universal
$K(A,n-1)$-bundle
\begin{equation*}
  \begin{xy}
    \morphism<925,0>[\W_{\bullet}K(A,n-1)`K(A,n)_{\bullet};]
  \end{xy}
\end{equation*}
is $G$-equivariant with respect to these actions. We define the
\emph{twisted universal bundle}
\begin{equation*}
  \begin{xy}
    \morphism<1000,0>[\W^G_{\bullet}K(A,n-1)`K^G(A,n)_{\bullet};]
  \end{xy}
\end{equation*}
by taking the homotopy quotient of the universal $K(A,n-1)$-bundle
with respect to the $G$-action.  An $A$-valued $n$-cocycle on
$\LW_{\bullet}G$ is now a commuting triangle
\begin{equation*}
  \begin{xy}
    \Vtriangle<600,300>[U_{\bullet}`K^G(A,n)_{\bullet}`\LW_{\bullet}G;``]
  \end{xy}
\end{equation*}
such that $U_{\bullet}\rightarrow\LW G_{\bullet}$ is a hypercover.
Equivalences of cocycles are defined analogously.  The analogue of
Corollary \ref{cor:coc} holds, and the construction proceeds just
as above.  The only difference is that, if we unpack the construction
of the Lie $2$-group $X_{\bullet}$ produced from this more general
notion of cocycle, we observe that instead of the $2$-simplices $X_2$
being the total space of a principal $A$-bundle, $X_2$ is now the
total space of a $G$-twisted principal $A$-bundle.  We leave the
remaining details to the interested reader.

The methods above additionally allow for the construction of Lie
$n$-groups for $n>2$.  In particular, there exists a finite
dimensional model of $\fivebrn$ as a Lie $7$-group extending $\strn$
by $K(\mathbb{Z},7)$.  We expect that a model of $\fivebrn$ also
exists as a Lie $6$-group extending $\strn$ by $K(U(1),6)$.

\bibliographystyle{amsplain}
\bibliography{master}

\end{document}